\documentclass{amsart}

\usepackage{amsmath,amssymb,amsfonts}
\usepackage[all]{xy}

\DeclareMathOperator{\rank}{rank}

\DeclareMathOperator{\Tor}{Tor}

\DeclareMathOperator{\Aut}{Aut}
\DeclareMathOperator{\End}{End}
\DeclareMathOperator{\id}{id}
\DeclareMathOperator{\Hom}{Hom}

\DeclareMathOperator{\im}{Im}
\DeclareMathOperator{\Ker}{Ker}
\DeclareMathOperator{\NS}{NS}
\DeclareMathOperator{\Trans}{T}
\DeclareMathOperator{\Inv}{T}
\DeclareMathOperator{\Orth}{S}
\DeclareMathOperator{\disc}{disc}
\DeclareMathOperator{\Sym}{Sym}

\DeclareMathOperator{\Tot}{Tot}
\DeclareMathOperator{\td}{td}

\newcommand{\vac}{|0\rangle}
\newcommand{\odd}{{\rm{odd}}}
\newcommand{\even}{{\rm{even}}}

\newcommand{\aG}{{\rm{a}}_G}
\newcommand{\mG}{{\rm{m}}_G}

\newcommand{\coloneqq}{:=}

\makeatletter

\newcommand{\Rmnum}[1]{\expandafter\@slowromancap\romannumeral #1@}
\makeatother

\newcommand{\RR}{\mathbf{R}}

\newcommand{\BG}{BG}
\newcommand{\EG}{EG}

\newcommand{\IC}{\mathbb{C}}
\newcommand{\IF}{\mathbb{F}}
\newcommand{\IK}{\mathbb{K}}
\newcommand{\IN}{\mathbb{N}}
\newcommand{\IP}{\mathbb{P}}
\newcommand{\IR}{\mathbb{R}}
\newcommand{\IQ}{\mathbb{Q}}
\newcommand{\IZ}{\mathbb{Z}}

\newcommand{\cF}{\mathcal{F}}

\newcommand{\cO}{\mathcal{O}}

\newcommand{\km}{\mathfrak{m}}
\newcommand{\kq}{\mathfrak{q}}

\newcommand{\lra}{\longrightarrow}

\theoremstyle{plain}
\newtheorem{theorem}{Theorem}[section]
\newtheorem{lemma}[theorem]{Lemma}
\newtheorem{proposition}[theorem]{Proposition}
\newtheorem{corollary}[theorem]{Corollary}
\theoremstyle{definition}
\newtheorem{definition}[theorem]{Definition}
\theoremstyle{remark}
\newtheorem{remark}[theorem]{Remark}

\begin{document}

\title[Smith theory and irreducible holomorphic symplectic manifolds]{Smith theory and irreducible holomorphic symplectic manifolds}

\author{Samuel Boissi\`ere}
\address{Samuel Boissi\`ere, Laboratoire de Math\'ematiques et Applications, UMR CNRS 6086, Universit\'e de Poitiers, T\'el\'eport 2, Boulevard Marie et Pierre Curie, F-86962 Futuroscope Chasseneuil}
\email{samuel.boissiere@math.univ-poitiers.fr}
\urladdr{http://www-math.sp2mi.univ-poitiers.fr/$\sim$sboissie/}

\author{Marc Nieper-Wi{\ss}kirchen}
\address{Marc Nieper-Wi{\ss}kirchen, Lehrstuhl f\"ur Algebra und Zahlentheorie, Universit\"ats\-stra{\ss}e~14, D-86159 Augsburg}
\email{marc.nieper-wisskirchen@math.uni-augsburg.de}
\urladdr{http://www.math.uni-augsburg.de/alg/}

\author{Alessandra Sarti}
\address{Alessandra Sarti, Laboratoire de Math\'ematiques et Applications, UMR CNRS 6086, Universit\'e de Poitiers, T\'el\'eport 2, Boulevard Marie et Pierre Curie, F-86962 Futuroscope Chasseneuil}
\email{alessandra.sarti@math.univ-poitiers.fr}
\urladdr{http://www-math.sp2mi.univ-poitiers.fr/$\sim$sarti/}

\date{\today}



\begin{abstract} 
We study the cohomological properties of the fixed locus $X^G$ of an automorphism group $G$ of prime order $p$ 
acting on a variety $X$ whose integral cohomology is torsion-free. We obtain a precise relation
between the mod $p$ cohomology of $X^G$ and natural invariants for the action of $G$
on the integral cohomology of $X$. We apply these results to irreducible holomorphic symplectic manifolds 
of deformation type of the Hilbert scheme of two points on a K3 surface:
the main result of this paper is a formula relating the dimension of the mod $p$ cohomology of 
$X^G$ with the rank and the discriminant of the invariant lattice in the second 
cohomology space with integer coefficients of $X$. 
\end{abstract}

\maketitle

%%%%%%%%%%%%%%%%%%%%%%%%%%%%%%%%%%%%%%%%%%%%%%%%%%%%%%%%%%%%%%%%%%
%%%%%%%%%%%%%%%%%%%%%%%%%%%%%%%%%%%%%%%%%%%%%%%%%%%%%%%%%%%%%%%%%%
%%%%%%%%%%%%%%%%%%%%%%%%%%%%%%%%%%%%%%%%%%%%%%%%%%%%%%%%%%%%%%%%%%

\section{Introduction}

Smith theory is the study of the cohomological properties of a group $G$ of prime order $p$
acting on a topological space $X$. The first important results were obtained by Smith 
in the late 1930's by the introduction of the so-called Smith cohomology groups and 
sequences (see Bredon~\cite{Bredon}). The use of equivariant cohomology to reformulate Smith theory was 
begun by Borel~\cite{Borel} in the 1950's and further formalized as the ``localisation
theorem'' of Borel--Atiyah--Segal--Quillen in the 1960's (see Dwyer--Wilkerson~\cite{DW}).

In this paper,
we use these ideas to relate the dimension of the mod $p$ cohomology of the fixed point
set $X^G$ to natural invariants for the action of $G$ on the integral cohomology
 $H^\ast(X,\IZ)$ for $2\leq p\leq 19$ (see Corollaries~\ref{cor:formulep=2}~\&~\ref{cor:formulepqqe}). This applies nicely to the study of prime order automorphisms
on some symplectic holomorphic varieties, particularly those in the deformation class
of the Hilbert scheme $S^{[2]}$ of two points on a K3 surface $S$. 
The first main result of this paper is a degeneracy condition for the spectral sequence of equivariant cohomology
$$
E_2^{r,s}\coloneqq H^r(G;H^s(X,\IF_p))\Longrightarrow H^{r+s}_G(X,\IF_p).
$$

\begin{theorem} Let $G$  be a group of prime order $p$ acting by automorphisms on an irreducible holomorphic symplectic  variety $X$. 
The spectral sequence of equivariant cohomology with coefficients in~$\IF_p$ degenerate at 
the $E_2$-term in the following cases:
\begin{enumerate}
\item $X$ is deformation equivalent to the Hilbert scheme $S^{[2]}$ of two points on a K3 surface $S$ and $p\notin \{2,5,23\}$.
\item $X=S^{[2]}$, $G$ acts by natural automorphisms (induced by automorphisms of the surface $S$) and $p\neq 2$.
\end{enumerate}
\end{theorem}

This result is proven in Proposition~\ref{prop:degenere} as a consequence of Deligne's criterium (see Section~\ref{ss:degE2})  applied to specific geometrical objects in the cohomology of $S^{[2]}$ (Lemma~\ref{lem:DeligneForE}, \ref{lem:DeligneForc2}~\&~\ref{lem:DeligneForq}).

For $X$ deformation equivalent to $S^{[2]}$, denote by $\Inv_G(X)\coloneqq H^2(X,\IZ)^G$ the invariant lattice and by $\Orth_G(X)\coloneqq \Inv_G(X)^\perp$ its orthogonal complement for the Beauville--Bogomolov bilinear form. We define (see Definitions~\ref{def:aG}~\&~\ref{def:mG}) two integers $\aG(X),\mG(X)\in\IN$ with the property that
$$
\frac{H^2(X,\IZ)}{\Inv_G(X)\oplus\Orth_G(X)}\cong\left(\frac{\IZ}{p\IZ}\right)^{\aG(X)}, \quad \rank\Orth_G(X)=\mG(X)(p-1).
$$

The second main result of this paper is the following formula:
\begin{theorem}
Let $X$ be deformation equivalent to $S^{[2]}$ and $G$ be a group of automorphisms of prime order $p$ on $X$ with
$3\leq p\leq 19$, $p\neq 5$. Then:
\begin{align*}
\dim H^\ast(X^G,\IF_p)&=324-2\aG(X)\left(25-\aG(X)\right)-(p-2)\,\mG(X)\left(25-2\aG(X)\right)\\
&+\frac{1}{2}\mG(X)\left((p-2)^2\mG(X)-p\right)
\end{align*}
with 
\begin{align*}
2&\leq (p-1)\mG(X)< 23,\\
0&\leq \aG(X)\leq \min\{(p-1)\mG(X),23-(p-1)\mG(X)\}.
\end{align*}
\end{theorem}

This formula is proven in Corollary~\ref{cor:mainformula}. The proof
uses first the localisation theorem as presented in Allday--Puppe~\cite{AlldayPuppe} (see Proposition~\ref{prop:XGTor}), secondly the degeneracy conditions for the spectral sequence of equivariant cohomology with coefficients
in $\IF_p$, then the determination of the $\IZ[G]$-module structure of the cohomology space $H^\ast(X,\IZ)$ (Proposition~\ref{prop:decompCurtisReiner}),
and finally the computation of the quotient $H^4(X,\IZ)/\Sym^2H^2(X,\IZ)$  (Proposition~\ref{prop:indiceS2H4}). The relation with the discriminant of the invariant lattice and its orthogonal is given in Lemma~\ref{lem:descriptionOrthInv}.

As an application of our results, we show in Section~\ref{ss:applications} that there are no free actions by finite groups on deformations of $S^{[2]}$, and we study  an order eleven automorphism on a Fano variety of lines of a cubic fourfold constructed by Mongardi~\cite{Mongardi2}.

{\it Aknowledgements.} We thank Olivier Debarre, Alexandru Dimca, William G. Dwyer, Viacheslav Kharlamov, Giovanni Mongardi, Kieran O'Grady and Volker Puppe for useful discussions and helpful comments.

%%%%%%%%%%%%%%%%%%%%%%%%%%%%%%%%%%%%%%%%%%%%%%%%%%%%%%%%%%%%%%%%%%

\section{Terminology and notation}\label{sec:terminology}

Let $p$ be a prime number and $G$ a finite cyclic group of order $p$.
We fix a generator~$g$ of~$G$. Put $\tau\coloneqq   g-1\in\IZ[G]$ and
$\sigma\coloneqq  1+g+\cdots+g^{p-1}\in\IZ[G]$.

Let $M$ be a finite-dimensional $\IF_p$-vector space equipped with a linear 
action of $G$ (a $\IF_p[G]$-module for short). The minimal polynomial of $g$, 
as an  endomorphism of $M$, divides the polynomial $X^p-1=(X-1)^p\in\IF_p[X]$ 
hence $g$ admits a Jordan normal form. We can thus decompose $M$ as a direct sum  
of some $G$-modules $N_q$ of dimension $q$ for $1\leq q\leq p$, where $g$ acts 
on $N_q$ by a matrix (in a suitable basis) of the following form:
$$
\left(
\raisebox{0.5\depth}{\xymatrixcolsep{1ex}\xymatrixrowsep{1ex}
\xymatrix{1 \ar@{-}[dddrrr] & 1 \ar@{-}[ddrr] &  & \text{\huge{$0$}} \\
 &   &  & \\
 & & & 1\\
\text{\huge{$0$}} & & & 1 }}
\right)
$$
Observe that $N_p$ is isomorphic to $\IF_p[G]$ as a  $G$-module. Throughout this paper, the notation $N_q$ will always denote the $\IF_p[G]$-module defined by the Jordan matrix of
 dimension $q$ above. We define the integer $\ell_q(M)$ as the number of blocks of length~$q$ in the Jordan decomposition of the $G$-module~$M$,
 in such a way that $M\cong\bigoplus_{q=1}^p N_q^{\oplus \ell_q(M)}$.

Let $H\coloneqq \bigoplus_{k\geq 0} H^k$ be a finite-dimensional graded $\IF_p$-vector space, where each graded component $H^k$ is equipped with a linear action of $G$. We define similarly, for any $k\geq 0$ and $1\leq q\leq p$, the integer $\ell_q^k(H)$ as the number of blocks of length~$q$ in the Jordan decomposition of the $G$-module $H^k$.

For any topological space $Y$ with the homotopy type of a finite CW-complex and any field $\IK$, we set $h^k(Y,\IK)\coloneqq \dim_\IK H^k(Y,\IK)$ and $h^*(Y,\IK)\coloneqq \sum_{k\geq 0} h^k(Y,\IK)$.

Let $X$ be a smooth connected orientable compact real even-dimensional manifold, with a smooth orientation-preserving 
action of~$G$. Denote by $X^G\subset X$ the fixed locus of $X$ for the action of $G$; then $X^G$ is a smooth submanifold of $X$. We define the integers $\ell_q^k(X)$ for $1\leq q\leq p$ and $0\leq k\leq\dim_\IR X$ as the number of blocks of length $q$ in the Jordan decomposition of the $G$-modules $H^k(X,\IF_p)$ and we set $\ell^\ast_q(X)\coloneqq \sum_{k\geq 0}\ell^k_q(X)$.

%%%%%%%%%%%%%%%%%%%%%%%%%%%%%%%%%%%%%%%%%%%%%%%%%%%%%%%%%%%%%%%%%%
%%%%%%%%%%%%%%%%%%%%%%%%%%%%%%%%%%%%%%%%%%%%%%%%%%%%%%%%%%%%%%%%%%
%%%%%%%%%%%%%%%%%%%%%%%%%%%%%%%%%%%%%%%%%%%%%%%%%%%%%%%%%%%%%%%%%%

\section{Some useful computations in group cohomology}

There is a projective resolution $F_*\xrightarrow{\epsilon}\IZ$ of $\IZ$ considered as a  $G$-module
with a trivial action, given by:
\begin{eqnarray}\label{res}
\cdots\lra\IZ[G]\xrightarrow{\tau}\IZ[G]\xrightarrow{\sigma}\IZ[G]\xrightarrow{\tau}\IZ[G]\xrightarrow{\epsilon}\IZ\lra 0
\end{eqnarray}
where $\epsilon$ is the summation map: $\epsilon(\sum_{j=0}^{p-1}\alpha_jg^j)=\sum_{j=0}^{p-1}\alpha_j$ and 
$\tau$, $\sigma$ act by multiplication.

Let $\IF_p\coloneqq \IZ/p\IZ$ considered as a trivial $G$-module. The cohomology groups of $G$ with coefficients in $\IF_p$ are the cohomology groups of the complex:
$$ 
0 \rightarrow\Hom_{G}(\IZ[G],\IF_p)\overset{\tau^\ast}{\rightarrow}\Hom_{G}(\IZ[G],\IF_p)\overset{\sigma^\ast}{\rightarrow}\Hom_{G}(\IZ[G],\IF_p)\overset{\tau^\ast}{\rightarrow}\cdots
$$
Observe that $\Hom_{G}(\IZ[G],\IF_p)\cong\IF_p$ by identifying a $G$-homomorphism $c$ with its image $c(1)\in \IF_p$, so
$\tau^*$ and $\sigma^*$ are identically zero and we get $H^i(G;\IF_p)\cong\IF_p$ for all $i\geq 0$.

Let now $M$ be as before a $\IF_p[G]$-module of finite dimension over $\IF_p$. The cohomology of $G$ with coefficients in $M$ can be computed in a similar way as the cohomology of the complex:
$$
0 \rightarrow M\overset{\bar\tau}{\rightarrow}M\overset{\bar\sigma}{\rightarrow}M\overset{\bar\tau}{\rightarrow}\cdots
$$ 
where $\bar\tau,\bar\sigma\in\IF_p[G]$ denote the reduction modulo $p$ of
$\tau$ and $\sigma$. Observe that $\bar\sigma=(\bar\tau)^{p-1}$. 
To compute $H^*(G;M)$ as an $\IF_p$-vector space it is enough to compute the groups $H^*(G;N_q)$. 

\begin{lemma}\label{lem:cohGadditif}\text{} 
\begin{enumerate}
\item If $q<p$ then $H^i(G;N_q)=\IF_p$ for all $i\geq 0$.
\item $H^0(G;N_p)=\IF_p$ and $H^i(G;N_p)=0$ for all $i\geq 1$.
\end{enumerate}
\end{lemma} 

\begin{proof}
The case $q=1$ is clear since $N_1\cong \IF_p$ as a trivial $G$-module. Assume now that $q\geq 2$.
Let $v_1,\ldots,v_q$ be a basis of $N_q$ such that $gv_1=v_1$ and $gv_i=v_{i-1}+v_i$ for all $i\geq 2$. It is
easy to compute that, as endomorphisms of $N_q$, one has $\ker(\bar\tau)=\langle v_1\rangle$ and $\im(\bar\tau)=\langle 
v_1,\ldots,v_{q-1}\rangle $ for all $q\leq p$. Using that $\bar\sigma=(\bar\tau)^{p-1}$ we get:
$$
\ker{\bar\sigma}=\begin{cases} 
N_q & \text{if } q<p,\\
\langle v_1,\ldots,v_{p-1}\rangle & \text{if } q=p,
\end{cases}\qquad \im(\bar\sigma)=\begin{cases}
0 & \text{if } q<p,\\
\langle v_1\rangle& \text{if } q=p.\\
\end{cases}
$$
If $p=q$ the result is clear. If $q<p$ it follows from:
$$
\frac{\ker(\bar\tau)}{\im(\bar\sigma)}\cong\langle v_1\rangle,\qquad\frac{\ker(\bar\sigma)}{\im(\bar\tau)}\cong\langle v_q\rangle.
$$
\end{proof}

Recall (see \cite[Ch. V]{Brown}) that the cohomology cross-product: 
$$
H^r(G;\IF_p)\otimes_{\IZ} H^s(G;M)\lra H^{r+s}(G\times G;\IF_p\otimes_{\IZ} M)
$$
followed by a diagonal approximation: 
$$
H^{r+s}(G\times G;\IF_p\otimes_\IZ M)\xrightarrow{\Delta^*} H^{r+s}(G;\IF_p\otimes_\IZ M)\cong H^{r+s}(G;M)
$$
defines a cup-product and a graded $H^*(G;\IF_p)$-module structure on $H^*(G;M)$, where $\IF_p\otimes_{\IZ} M$ is
considered as a  $G$-module for the diagonal action (and is isomorphic to $M$ as a $G$-module since $G$ acts
trivially on $\IF_p$). Here the diagonal approximation $\Delta^*$ is induced by the maps $\Delta_{r,s}: F_{r+s}\lra F_r\otimes F_s$ given by:
$$
\Delta_{r,s}(1)=\begin{cases}
1\otimes 1&\text{for } r \text{ even}\\
1\otimes g&\text{for } r \text{ odd, }s \text{ even}\\
\sum_{0\leq i<j\leq p-1} g^i\otimes g^j&\text{for } r\text{ odd, } s\text{ odd}
\end{cases}
$$

Let $\alpha\in H^r(G;\IF_p)$ and $\beta\in H^s(G;M)$. Using again the natural identifications $\Hom_G(F_r,\IF_p)\cong\IF_p$ and $\Hom_G(F_s,M)\cong M$ one computes easily the cup-product $\alpha\cup\beta$ as follows:
\par{\it (i)} If $r$ is even, $\alpha\cup\beta=\alpha\beta$.
\par{\it (ii)} If $r$ is odd and $s$ is even, one has $\bar\tau(\beta)=0$ (see the proof of Lemma~\ref{lem:cohGadditif}) so $g\beta=\beta$ and $\alpha\cup\beta=\alpha\beta$.
\par{\it (iii)} If $r$ is odd and $s$ is odd, 
$$
\alpha\cup\beta=\alpha\cdot\left((g+2g^2+\cdots+(p-1)g^{p-1})\beta\right).
$$
We study the action of $g+2g^2+\cdots+(p-1)g^{p-1}$ on $N_q$ for $1\leq q\leq p$.

\begin{lemma}\label{lem:calcul} As an endomorphism of $N_q$, with $1\leq q\leq p$, one has:
$$
g+2g^2+\cdots+(p-1)g^{p-1}=\begin{cases}
0&\text{if } q\leq p-2,\\
-\bar\tau^{q-1} & \text{if } q=p-1,\\
-\bar\tau^{q-1}-\bar\tau^{q-2} & \text{if } q=p.
\end{cases}
$$
\end{lemma}

\begin{proof} One computes:
\begin{align*}
\sum_{i=1}^{p-1}ig^i&=\sum_{i=1}^{p-1}\sum_{j=0}^i i\binom{i}{j}\bar\tau^j
=\sum_{j=0}^{p-1}\left(\sum_{i=j}^{p-1}i\binom{i}{j}\right)\bar\tau^j
=\sum_{j=0}^{p-1}\left(\sum_{k=0}^{p-1-j}(j+k)\binom{j}{j+k}\right)\bar\tau^j\\
&=\sum_{j=0}^{p-1}\left(j\binom{p}{j+1}+(j+1)\binom{p}{j+2}\right)\bar\tau^j
\end{align*}
where the last equality follows from an easy induction on $p$ (for any integer $p$). By reduction modulo $p$,
all binomial coefficients $\binom{p}{\ell}$ vanish for $1\leq\ell\leq p-1$ so:
$$
\sum_{i=1}^{p-1}ig^i=-\bar\tau^{p-1}-\bar\tau^{p-2}.
$$
Since $\bar\tau^q=0$ on $N_q$, the result follows.
\end{proof}

In the special case $M=\IF_p$, in case {\it (iii)} one obtains $\alpha\cup\beta=\alpha\beta$ if $p=2$ and $\alpha\cup\beta=0$ if $p\geq 3$. It follows that, as a  graded commutative algebra:
$$
H^*(G;\IF_p)\cong\begin{cases}
\IF_p[u] & \text{if } p=2,\\
\Lambda^*(s)\otimes_{\IF_p}\IF_p[t] & \text{if } p\geq 3,
\end{cases}
$$
where $\deg(u)=1$, $\deg(s)=1$, $\deg(t)=2$ and $\Lambda^*(s)$ denotes the exterior algebra over $\IF_p$ generated by $s$ (see \cite[Proposition~1.4.2]{AlldayPuppe}). 

\begin{proposition} \label{prop:cohGmodule}
$H^*(G;N_p)\cong N_p^G\cong\IF_p$ is a trivial $H^*(G;\IF_p)$-module. For~${q<p}$, $H^*(G;N_q)$ is a free $H^*(G;\IF_p)$-module generated by $H^0(G;N_q)\cong\IF_p$.
\end{proposition}

\begin{proof}
This follows from Lemma~\ref{lem:cohGadditif} and the discussion above. The cases $q=p$  or $p=2$ are clear. In the case $p\geq 3$ and $q<p$, for $\alpha\in H^r(G;\IF_p)$ and $\beta\in H^s(G;M)$ with $r$ odd and $s$ odd, following the notation used in the proof of Lemma~\ref{lem:cohGadditif}, $\beta$~can be represented by a class $v_q$ with $\bar\sigma v_q=0$. Since $\bar\sigma=\bar\tau^{p-1}$, using Lemma~\ref{lem:calcul} one gets $\alpha\cup\beta=0$ in case ${\it (iii)}$. 
The result follows.
\end{proof}

We denote by $R$ the polynomial part of $H^*(G;\IF_p)$ (that is: $R=\IF_p[u]$ for $p=2$ and $R=\IF_p[t]$ for $p\geq 3$).
We consider $\IF_p$ as a $R$-module by evaluating at zero (setting $u=0$ for $p=2$ and $t=0$ for $p\geq 3$). For any $1\leq q\leq p$, we consider $H^*(G;N_q)$ as a $R$-module by the inclusion $R\hookrightarrow H^*(G;\IF_p)$.

\begin{corollary}\label{cor:dimtor}\text{}
\begin{enumerate}
\item For $p=2$ and $q<p$, one has $\dim_{\IF_p}\Tor_0^R(H^*(G;N_q),\IF_p)=1$ and for $i>0$, $\Tor_i^R(H^*(G;N_q),\IF_p)=0$.

\item For $p\geq 3$ and $q<p$, one has $\dim_{\IF_p}\Tor_0^R(H^*(G;N_q),\IF_p)=2$ and for $i>0$, $\Tor_i^R(H^*(G;N_q),\IF_p)=0$.

\item For $p\geq 2$, one has: 
$$
\dim_{\IF_p}\Tor_0^R(H^*(G;N_p),\IF_p)=1=\dim_{\IF_p}\Tor_1^R(H^*(G;N_p),\IF_p)=1,
$$
and for $i\geq 2$, $\Tor_i^R(H^*(G;N_p),\IF_p)=0$.
\end{enumerate}
\end{corollary}

\begin{proof}\text{}
There is a length 2 projective resolution of $\IF_p$ as a $R$-module given by:
$$
0\lra R\overset{\phi}{\lra} R\lra\IF_p\lra 0
$$
where $\phi\colon R\to R$ is the multiplication by $u$ for $p=2$ and by $t$ for $p\geq 3$, so $\Tor_i^R(H^*(G;N_q),\IF_p)=0$ for $i\geq 2$ and $q\leq p$.
\par\noindent (a) Assume that $q<p$. By Proposition~\ref{prop:cohGmodule}, $H^*(G;N_q)$ is a free {$R$-module} so $\Tor_i^R(H^*(G;N_q),\IF_p)=\{0\}$ for $i\geq 1$. Recall that: 
$$
\Tor_0^R(H^*(G;N_q),\IF_p)\cong H^*(G;N_q)\otimes_R\IF_p.
$$
For $p=2$, $H^*(G;N_q)$ is generated by any non zero element $v\in H^0(G;N_q)$ as a $R$-module so $\dim_{\IF_p} H^*(G;N_q)\otimes_R\IF_p=1$; for $p\geq 3$, $H^*(G;N_q)$ is again generated by any non zero $v\in H^*(G;N_q)$ as a $H^*(G;\IF_p)$-module, so is generated by $v$ and~$sv$ as a  $R$-module: this gives $\dim_{\IF_p} H^*(G;N_q)\otimes_R\IF_p=2$. 
\par\noindent (b) Take $q=p$. From the length 2 resolution of $\IF_p$ as a  $R$-module, using  Proposition~\ref{prop:cohGmodule} one gets: 
$$
\Tor_1^R(H^*(G;N_p),\IF_p)\cong \ker(\phi\colon H^*(G;N_p)\to H^*(G;N_p))=H^*(G;N_p).
$$ 
By Lemma~\ref{lem:cohGadditif} this space is one-dimensional so: 
$$
\Tor_0^R(H^*(G;N_p),\IF_p)\cong H^*(G;N_p)\otimes_R\IF_p\cong \IF_p.
$$
\end{proof}

%%%%%%%%%%%%%%%%%%%%%%%%%%%%%%%%%%%%%%%%%%%%%%%%%%%%%%%%%%%%%%%%%%
%%%%%%%%%%%%%%%%%%%%%%%%%%%%%%%%%%%%%%%%%%%%%%%%%%%%%%%%%%%%%%%%%%
%%%%%%%%%%%%%%%%%%%%%%%%%%%%%%%%%%%%%%%%%%%%%%%%%%%%%%%%%%%%%%%%%%

\section{Equivariant cohomology}

\subsection{Basic facts on equivariant cohomology}

Let $\EG\to \BG$ be a universal $G$-bundle in the category of CW-complexes. Denote by $X_G\coloneqq \EG\times_G X$ the orbit space for the diagonal action of $G$ on the product $\EG\times X$ and $f\colon X_G\to \BG$ the map induced by the projection onto the first factor. The map $f$ is a locally trivial fibre bundle with typical fibre $X$ and structure group $G$. The \emph{equivariant cohomology} of the pair $(X,G)$ with coefficients in $\IF_p$ is defined by ${H^*_G(X,\IF_p)\coloneqq H^*(X_G,\IF_p)}$, naturally endowed with a graded $H^*(\BG,\IF_p)$-module structure.
 Note that there is an isomorphism of graded algebras $H^*(\BG,\IF_p)\cong H^*(G;\IF_p)$.
The Leray--Serre spectral sequence associated to the map $f$ gives a spectral sequence converging to the equivariant cohomology with coefficients in $\IF_p$:
$$
E_2^{r,s}\coloneqq H^r(G;H^s(X,\IF_p))\Longrightarrow H^{r+s}_G(X,\IF_p).
$$

\begin{remark}
By assumption $X$ has the homotopy type of a finite $G$-CW-complex.
Denote by $C^*(X)$ the cellular cochain complex of $X$ with coefficients in
 $\IF_p$. The spaces $C^s(X)$ 
are finitely dimensional $\IF_p$-vector spaces. 
Recall that $\varepsilon\colon F_*\to\IZ$ denotes 
the projective resolution of $\IZ$ as a trivial $\IZ[G]$-module, and define the double 
complex $\beta_G^{r,s}(X)\coloneqq \Hom_G(F_r,C^s(X))$. 
As $C^*(X)$ is quasi-isomorphic to $R\Gamma(X,\IF_p)$ in the derived category of $G$-modules, the cohomology of the total complex, the cohomology of the 
total complex $\Tot\beta_G(X)$ computes the equivariant 
cohomology (see Allday--Puppe~\cite[Theorem~1.2.8]{AlldayPuppe}): $H^*_G(X,\IF_p)\cong H^*(\Tot\beta_G(X))$. This yields a concrete description of the 
first quadrant spectral sequence converging to the  equivariant cohomology.
\end{remark}

\subsection{Cohomology of the fixed locus}

Recall that $R$ denotes the polynomial part of $H^\ast(G;\IF_p)$. We prove the following formula (see Allday--Puppe~\cite{AlldayPuppe} for related results):

\begin{proposition}\label{prop:XGTor} For $p\geq 2$ one has:
$$
h^\ast(X^G,\IF_p)=\nu\cdot\left(\dim_{\IF_p}\Tor^R_0(H^*_G(X,\IF_p),\IF_p)-\dim_{\IF_p}\Tor^R_1(H^*_G(X,\IF_p),\IF_p)\right)
$$
with $\nu=1$ for $p=2$ and $\nu=\frac{1}{2}$ for $p\geq 3$.
\end{proposition}

\begin{proof}
The graded $R$-module $H^\ast_G(X,\IF_p)$ is of finite type so it admits a minimal free resolution~\cite[Proposition~A.4.12]{AlldayPuppe}:
$$
0\lra L_1\lra L_0\lra H^\ast_G(X,\IF_p)\lra 0
$$
such that $\rank_R L_i=\dim_{\IF_p} \Tor^R_i(H^\ast_G(X,\IF_p),\IF_p)$. Write $R=\IF_p[T]$ (with $T=u$ 
of degree one if $p=2$ and $T=t$ of degree two if $p\geq 3$). For $\alpha\in\IF_p$, 
define $\IF_{p,\alpha}\coloneqq R/(T-\alpha)$. This is consistent with the previous description 
$\IF_p\cong\IF_{p,0}$ as an $R$-module.
 For $\alpha\neq 0$, the functor $-\otimes_R\IF_{p,\alpha}$ is exact~\cite[Lemma~A.7.2]{AlldayPuppe} so:
\begin{align*}
\dim_{\IF_p} H^\ast_G(X,\IF_p)\otimes_R\IF_{p,\alpha}&=\dim_{\IF_p} L_0\otimes_R\IF_{p,\alpha}-\dim_{\IF_p} L_1\otimes_R\IF_{p,\alpha}\\
&=\rank_RL_0-\rank_R L_1\\
&=\dim_{\IF_p}\Tor^R_0(H^\ast_G(X,\IF_p),\IF_p)\\
&\qquad-\dim_{\IF_p}\Tor^R_1(H^\ast_G(X,\IF_p),\IF_p). 
\end{align*}
For $\alpha\neq 0$, one has
 $H^\ast_G(X,\IF_p)\otimes_R\IF_{p,\alpha}\cong H^\ast(\beta_G(X)\otimes_R\IF_{p,\alpha})$
(this cohomology is computed with the total differential). We now use the following analogue of the localisation theorem in equivariant cohomology~\cite[Theorem~1.3.5, Theorem~1.4.5]{AlldayPuppe}: for $\alpha\neq 0$, one has 
$$
H^\ast(\beta_G(X)\otimes_R\IF_{p,\alpha})\cong
\begin{cases}
H^*(X^G,\IF_p)&\text{if } p=2\\
H^*(X^G,\IF_p)\otimes_{\IF_p}\Lambda(s)&\text{if } p\geq 3.
\end{cases}
$$
The result follows.
\end{proof}

If the spectral sequence of equivariant cohomology with coefficients in $\IF_p$ degenerates at the $E_2$-term, it induces
 an isomorphism of graded $H^*(G;\IF_p)$-modules:
$$
H^*(G;H^*(X,\IF_p))\cong H^*_G(X,\IF_p).
$$
Using Corollary~\ref{cor:dimtor}, Proposition~\ref{prop:XGTor} gives immediately:

\begin{corollary}\label{cor:XGlq} If the spectral sequence of equivariant cohomology with coefficients in $\IF_p$ degenerates at the $E_2$-term, then for $p\geq 2$ one has:
$$
h^\ast(X^G,\IF_p)=\sum_{1\leq q<p}\ell^\ast_q(X).
$$
\end{corollary}

This formula can be stated differently, using only the parameter 
$\ell^\ast_p(X)$, that will appear to be the most important in the sequel: 

\begin{corollary}
If the spectral sequence of equivariant cohomology with coefficients in $\IF_p$ 
degenerates at the $E_2$-term, then for $p\geq 2$ one has:
$$
h^\ast(X^G,\IF_p)=\dim_{\IF_p}H^\ast(X,\IF_p)^G-\ell^\ast_p(X).
$$
\end{corollary}

\begin{proof}
Since each Jordan block of $H^\ast(X,\IF_p)$ contains a one-dimensional 
invariant subspace, one gets $\dim_{\IF_p} H^*(X,\IF_p)^G=\sum_{1\leq q\leq p}\ell^\ast_q(X)$.
One conludes by using Corollary~\ref{cor:XGlq}.
\end{proof}

\subsection{Degeneracy condition of the spectral sequence}
\label{ss:degE2}
Even under very nice conditions, one can not expect the collapsing of the spectral sequence in general.
For instance, take $X$  a non-singular, real projective algebraic variety and $g\colon X(\IC)\to X(\IC)$
the involution of complex conjugation, $G=\{1,g\}$ the order two group acting on $X(\IC)$. Then $X$ is
called a \emph{GM-variety} if the spectral sequence of equivariant cohomology with $E_2^{r,s}=H^r(G,H^s(X(\IC),\IF_2))$
degenerates. See Krasnov~\cite{Krasnov1,Krasnov2} for some examples of $GM$ and non-$GM$ varieties. In this section, we
prove some degeneracy conditions that will be useful for certain symplectic holomorphic varieties.

\begin{proposition}\label{prop:degE2K3} Assume that $\dim_\IR X=4$ and $H^\odd(X,\IF_p)=0$. If $X$ has a fixed point for the action of $G$, then the spectral sequence of equivariant cohomology with coefficients in $\IF_p$ degenerates at the $E_2$-term.
\end{proposition}

\begin{proof}
Let $x\in X$ be a fixed point for $G$. 
It induces a section $s\colon \BG\to X_G$ of the projection~$f\colon X_G\to \BG$. Denote by
$$
u\coloneqq s_\ast 1\in H^4(X_G,\IF_p)
$$ 
the proper push-forward of the unit in $H^\ast(\BG,\IF_p)$. We can view $u$ as a morphism $u\colon\IF_p\to\IF_p[4]$ in the derived category of sheaves of $\IF_p$-vector spaces over~$X_G$. Pushing down yields a morphism
$$
q\coloneqq \RR f_\ast u\colon \RR f_\ast\IF_p\to\RR f_\ast\IF_p[4]
$$
in the corresponding derived category of sheaves over $\BG$. From Deligne~\cite[Proposition~2.1]{Deligne} modified by the arguments 
of~\cite[Remarque~(1.9), $s=2$]{Deligne} (where we use the assumption $H^\odd(X,\IF_p)=0$) we get that 
if $q\colon R^0f_\ast\IF_p\to R^4 f_\ast\IF_p$ is an isomorphism, then $\IF_p$ satisfies the 
Lefschetz condition relative to $u$, that is: 
$$
\RR f_\ast\IF_p\cong\bigoplus_i R^i f_\ast\IF_p[-i]
$$
and the spectral sequence of equivariant cohomology with coefficients in $\IF_p$ degenerates at the $E_2$-term.

In order to show that  $q$ is an isomorphism, note that its source and target, being higher direct 
images of a constant sheaf along a locally trivial fibration, are locally constant 
sheaves. Thus it is enough to show that $q$ is an isomorphism fibre-wise. This follows from base change, as the fibre of $R^if_\ast\IF_p$ at a point $t\in \BG$ is just $H^i(X,\IF_p)$ and the fibre of the morphism $q$ at $t$ 
is the multiplication by the fundamental class $[x]\in H^4(X,\IF_p)$ of the fixed point $x$. 
\end{proof}

Let $\cF$ be a vector bundle on $X$. Recall that a \emph{$G$-linearisation of $\cF$} is given by the data of
homomorphisms $\phi_g\colon g^\ast \cF \to \cF$
for all $g \in G$ such that the cocycle condition
$\phi_h \circ h^\ast(\phi_g) = \phi_{hg}\colon g^\ast h^\ast \cF \to \cF$
is fulfilled for all $g,h\in G$. A \emph{$G$-linearised vector bundle}
is a vector bundle together with the data of a $G$-linearisation. Note that $G$-equivariant
resolutions exist, see Elagin \cite{Elagin}. Natural examples
are  the (co)tangent bundle on
 $X$ (where the $G$-linearization is given by pullback along the
action of $G$) or the sheaf of section $\cO(D)$ for any divisor $D$ on~$X$ that is globally invariant
for the action of $G$.

A $G$-linearisation on a vector bundle $\cF$ induces an ordinary $G$-action
on the \'etale space of $\cF$, which we denote by $\cF$ again, such that the
natural projection $\cF \to X$ becomes a $G$-equivariant map. We can then form
the space $\cF_G \coloneqq  \cF \times_G \EG$, which has a natural map to $X_G$,
making it canonically into a vector bundle over $X_G$. If we restrict
$\cF_G$ to a fibre of $f\colon X_G \to \BG$ (all of which are isomorphic to $X$),
it becomes the vector bundle $\cF$ over $X$ again.

If $\cF$ has the additional structure of a complex vector bundle and the
$G$-linearisa\-tion of $\cF$ is compatible with this structure, the induced
bundle $\cF_G$ inherits this structure as a complex vector bundle. Given two
$G$-linearised (complex) vector bundles $\cF$ and $\cF'$ over $X$, there is the obvious
notion of a \emph{$G$-equivariant homomorphism between $\cF$ and $\cF'$}. It
induces naturally an ordinary homomorphism between $\cF_G$ and $\cF'_G$. This
construction is compatible with the notion of exact sequences, so we get in
fact a group homomorphism from the \emph{$G$-equivariant Grothendieck group}
$K^0_G(X)$ of $X$ to the ordinary Grothendieck group $K^0(X_G)$ of (complex)
vector bundles. By forgetting the $G$-linearisations, one defines another
group homomorphism from $K^0_G(X) \to K^0(X)$.

This allows one to construct classes in the equivariant cohomology. Let
$\alpha$ be a characteristic class of complex
$K$-theory with values in $\IF_p$ (in the sequel, we will use reductions modulo $p$ of
integral characteristic classes like integral linear combinations of Chern classes). Let $\cF$ 
be a $G$-equivariant vector bundle, or more generally a class in the $G$-equivariant 
Grothendieck group $K^0_G(X)$.
Then $\alpha(\cF_G) \in H^\ast(X_G, \IF_p)$. By the naturality
of characteristic classes, the restriction of $\alpha(\cF_G)$ to a fibre 
of $f\colon X_G \to \BG$ is just $\alpha(\cF)$.

\begin{proposition}\label{prop:degE2Hilb2} 
Assume that $\dim_\IR X=8$ and $H^\odd(X,\IF_p)=0$. Let $\cF\in K^0_G(X)$ be a class
in the equivariant complex $K$-theory of $X$ and 
$c\coloneqq \alpha_2(\cF)\in H^4(X,\IF_p)$
a characteristic class. Assume that the multiplication maps
\begin{align*}
H^2(X,\IF_p) &\to H^6(X,\IF_p),\quad \beta \mapsto c \cup \beta\\
H^0(X,\IF_p) &\to H^8(X,\IF_p),\quad \beta \mapsto c^2 \cup \beta
\end{align*}
are isomorphisms. Then the spectral sequence of equivariant cohomology with 
coefficients in $\IF_p$ degenerates at the $E_2$-term.
\end{proposition}

\begin{proof}
The proof is virtually the same as for  proposition~\ref{prop:degE2K3}. Denoting
as above ${u\coloneqq \alpha(\cF_G)\in H^4(X_G,\IF_p)}$ and 
$q\coloneqq \RR f_\ast u\colon \RR f_\ast\IF_p\to\RR f_\ast\IF_p[4]$, we use again
Deligne \cite[Proposition (2.1)]{Deligne} modified by the arguments of
\cite[Remarque (1.9), $s = 2$]{Deligne}: if $q\colon R^2f_\ast\IF_p\to R^6 f_\ast\IF_p$
and $q^2\colon R^0f_\ast\IF_p\to R^8 f_\ast\IF_p$ are isomorphisms, then $\IF_p$ satisfies
the Lefschetz condition relative to $u$ and the spectral sequence degenerates at the $E_2$-term.
Again this can be checked fibrewise, where these maps are the multiplications by $c$ and 
$c^2$ respectively.
\end{proof}

\begin{remark}
As an example, assume that $X$ is a smooth complex algebraic variety of complex dimension two 
that possesses a $G$-fixed point $x$. The skyscraper sheaf
to this point defines a class $[x]$ in $K^0_G(X)$ (after a finite $G$-equivariant resolution). For $\alpha$ take the second
Chern class~$c_2$. It follows that $u \coloneqq c_2([x]) \in H^4(X_G, \IF_p)$
is a class whose restriction to each fibre $X$ of $p\colon X_G \to \BG$ is just
the fundamental class of the point $x$ since $\dim_\IR X=4$. This is the class used in 
Proposition~\ref{prop:degE2K3}.
\end{remark}

\begin{remark} The preceding two propositions are valid for any finite group $G$ acting on $X$, not only $\IZ/p\IZ$.
\end{remark}

%%%%%%%%%%%%%%%%%%%%%%%%%%%%%%%%%%%%%%%%%%%%%%%%%%%%%%%%%%%%%%%%%%
%%%%%%%%%%%%%%%%%%%%%%%%%%%%%%%%%%%%%%%%%%%%%%%%%%%%%%%%%%%%%%%%%%
%%%%%%%%%%%%%%%%%%%%%%%%%%%%%%%%%%%%%%%%%%%%%%%%%%%%%%%%%%%%%%%%%%

\section{Two integral parameters}

Assume that $H^*(X,\IZ)$ is torsion-free.  
By the universal coefficient theorem, one has $H^*(X,\IF_p)\cong H^*(X,\IZ)\otimes_\IZ \IF_p$ and the homomorphisms of 
reduction modulo~$p$, denoted by
$$
\kappa_k\colon H^k(X,\IZ)\lra H^k(X,\IF_p)
$$ 
are surjective for all $k$.

Let $\xi_p$ be a primitive $p$-th root of the unity, $K\coloneqq \IQ(\xi_p)$ 
and $\mathcal{O}_K\coloneqq\IZ[\xi_p]$ the ring of algebraic integers of $K$. 
By a classical theorem of Masley--Montgomery~\cite{MasleyMontgomery}, $\mathcal{O}_K$ is a PID if and only 
if $p\leq 19$. 
The $G$-module structure of $\mathcal{O}_K$ is defined by $g\cdot x=\xi_p x$ for $x\in \mathcal{O}_K$. For any $a\in \mathcal{O}_K$, we
denote by  $(\mathcal{O}_K,a)$ the module $\mathcal{O}_K\oplus \IZ$ whose $G$-module structure is defined by 
$g\cdot(x,k)=(\xi_p x+ka,k)$.

\begin{proposition} 
\label{prop:decompCurtisReiner}
Assume that $H^*(X,\IZ)$ is torsion-free and $3\leq p\leq 19$. Then for $0\leq k\leq \dim_\IR X$ one has:
\begin{enumerate}
\item\label{propCurtis:item1} $\ell_i^k(X)=0$ for $2\leq i\leq p-2$.
\item\label{propCurtis:item2} $\rank_{\IZ} H^k(X,\IZ)=p\ell_p^k(X)+(p-1)\ell_{p-1}^k(X)+\ell_1^k(X)$.
\item\label{propCurtis:item3} $\dim_{\IF_p} H^k(X,\IF_p)^G=\ell_p^k(X)+\ell_{p-1}^k(X)+\ell_1^k(X)$.
\item\label{propCurtis:item4} $\rank_\IZ H^k(X,\IZ)^G=\ell_p^k(X)+\ell_1^k(X)$.
\end{enumerate}
\end{proposition}

\begin{proof}
By a theorem of Diederichsen and Reiner~\cite[Theorem 74.3]{CurtisReiner}, $H^k(X,\IZ)$ is isomorphic as a {$\IZ[G]$-module} to a direct 
sum:
$$
(A_1,a_1)\oplus\cdots\oplus (A_r,a_r)\oplus A_{r+1}\oplus \cdots\oplus A_{r+s}\oplus Y
$$
where the $A_i$ are fractional ideals in $K$, $a_i\in A_i$ are such that $a_i\notin (\xi_p-1)A_i$ and
 $Y$ is a free $\IZ$-module of finite rank on which $G$ acts trivially. The $G$-module structure on
 $A_i$ is defined by $g\cdot x=\xi_p x$ for all $x\in A_i$, and $(A_i,a_i)$ denotes the module
 $A_i\oplus \IZ$ whose $G$-module structure is defined by $g\cdot(x,k)=(\xi_p x+ka_i,k)$. Since $\mathcal{O}_K$ is a PID, there is only one ideal class
 in $K$ so we have an isomorphism of
 $\IZ[G]$-modules:
$$
H^k(X,\IZ)\cong \oplus_{i=1}^r(\mathcal{O}_K,a_i)\oplus \mathcal{O}_K^{\oplus s}\oplus \IZ^{\oplus t}
$$
for some $a_i\notin(\xi_p-1)\mathcal{O}_K$. The matrix of the action of $g$ acting on $\mathcal{O}_K$ is:
$$
\left(
\raisebox{0.5\depth}{\xymatrixcolsep{1ex}\xymatrixrowsep{1ex}
\xymatrix{0 \ar@{-}[ddrr] & &  \text{\huge{$0$}} & -1\ar@{-}[ddd] \\
1\ar@{-}[ddrr] &   &  & \\
 & &  0 & \\
\text{\huge{$0$}} & & 1 & -1 }}
\right)
$$
so its minimal polynomial over $\IQ$ is the cyclotomic polynomial $\Phi_p$, hence $\mathcal{O}_K$ has no {$G$-invariant} element over $\IZ$. 
Over $\IF_p$, the minimal polynomial of $\mathcal{O}_K\otimes_\IZ\IF_p$ is $\Phi_p(X)=(X-1)^{p-1}$, so $\mathcal{O}_K\otimes_\IZ\IF_p$ is isomorphic to $N_{p-1}$ as a $\IF_p[G]$-module. 
The matrix of the action of $g$ on $(\mathcal{O}_K,a)$  is:
$$
\left(
\raisebox{0.5\depth}{\xymatrixcolsep{1ex}\xymatrixrowsep{1ex}
\xymatrix{0 \ar@{-}[ddrr] & &  \text{\huge{$0$}} & -1\ar@{-}[ddd] & \star\ar@{-}[ddd]\\
1\ar@{-}[ddrr] &  &  & \\
 & &  0 & \\
\text{\huge{$0$}} & & 1 & -1  & \star\\
0\ar@{-}[rrr] & & & 0 & 1}}
\right)
$$
so its minimal polynomial over $\IQ$ is $(X-1)\Phi_p(X)=X^p-1$, hence the subspace of invariants $(\mathcal{O}_K,a)^G$ is one-dimensional.
Over $\IF_p$, the minimal polynomial of $(\mathcal{O}_K,a)\otimes_\IZ\IF_p$ is  $(X-1)^p$,  so $(\mathcal{O}_K,a)\otimes_\IZ\IF_p$ is isomorphic to $N_p\cong\IF_p[G]$ as a $\IF_p[G]$-module. By reduction modulo $p$, the universal coefficient theorem implies:
$$
H^k(X,\IF_p)\cong N_{p}^{\oplus r}\oplus N_{p-1}^{\oplus s}\oplus N_1^{\oplus t}
$$
as $\IF_p[G]$-modules, so $\ell_{p}^k(X)=r$, $\ell_{p-1}^k(X)=s$, $\ell_1^k(X)=t$ and $\ell_i^k(X)=0$ for $2\leq i\leq p-2$, this proves (\ref{propCurtis:item1}) and (\ref{propCurtis:item2}). Since each block contains a one-dimensional $G$-invariant subspace, this implies also that: 
$$
\dim_{\IF_p}H^k(X,\IF_p)^G=\ell_{p}^k(X)+\ell_{p-1}^k(X)+\ell_1^k(X),
$$
this proves (\ref{propCurtis:item3}). Over $\IZ$, only the trivial $G$-module in $H^k(X;\IZ)$ and the $G$-modules $(\mathcal{O}_K,a)$ contain a $G$-invariant subspace, of dimension $1$, so: 
$$
\rank_\IZ H^k(X,\IZ)^G=r+t=\ell_p^k(X)+\ell_1^k(X),
$$
this proves (\ref{propCurtis:item4}).
\end{proof}

\begin{remark}
\label{rem:decompCurtisReinerp=2}
Assume that $H^*(X,\IZ)$ is torsion-free and $p=2$. The above argument is much more basic and one gets 
easily, with the same notation, 
that $\ell_1^k(X)=s+t$, $\ell_2^k(X)=r$, $\rank_\IZ H^k(X,\IZ)^G=r+t$, $\dim_{\IF_2}H^k(X,\IF_2)^G=\ell_1^k(X)+\ell_2^k(X)$.
\end{remark}

Recall that $G=\langle g\rangle$, $\tau=g-1\in\IZ[G]$ and $\sigma=1+g+\cdots+g^{p-1}\in\IZ[G]$. We denote also
by $g,\tau,\sigma$ their actions on any $\IZ[G]$-module.
For $1\leq k \leq \dim_\IR X-1$ we define:
\begin{align*}
\Inv_G^k(X)&\coloneqq \ker(\tau)\cap H^{k}(X,\IZ),\\
\Orth_G^k(X)&\coloneqq \ker(\sigma)\cap H^{k}(X,\IZ).
\end{align*}
As kernels these modules are primitive in $H^k(X,\IZ)$. 

\begin{lemma}\label{lem:H2STisptorsion}
Assume that $H^\ast(X,\IZ)$ is torsion-free and $2\leq p\leq 19$. Then for all~$k$, 
$\frac{H^k(X,\IZ)}{\Inv_G^k(X)\oplus\Orth_G^k(X)}$ is a $p$-torsion module.
\end{lemma}

\begin{proof} First observe that $\Inv_G^k(X)\cap\Orth_G^k(X)=\{0\}$ since $H^k(X,\IZ)$ 
has no $p$-torsion. As in the proof of Proposition~\ref{prop:decompCurtisReiner}, for each $k$ one has a
 $\IZ[G]$-module  decomposition:
$$
H^k(X,\IZ)\cong \oplus_{i=1}^r(\mathcal{O}_K,a_i)\oplus \mathcal{O}_K^{\oplus s}\oplus\IZ^{\oplus t}.
$$
It is clear that $\IZ^{\oplus t}\subset\Inv_G^k(X)$ and $\mathcal{O}_K^{\oplus s}\subset\Orth_G^k(X)$.
In any term $(\mathcal{O}_K,a)= \mathcal{O}_K\oplus\IZ$, denoting $v\coloneqq (0,1)$ in this decomposition,
we show that $pv\in\Inv_G^k(X)\oplus\Orth_G^k(X)$.
For this, observe that the quotient of $\mathcal{O}_K$ by its maximal ideal $(\xi_p-1)$ is $\IZ/p\IZ$,
so for any $x\in \mathcal{O}_K$ there exists $z\in \mathcal{O}_K$ such that $px=(\xi_p-1)z$. One has $\tau(v)=(a,0)$ hence there
exists $z\in \mathcal{O}_K$ such that $\tau(pv)=(pa,0)=((\xi_p-1)z,0)$. Now $\tau((z,0))=((\xi_p-1)z,0)$ hence
$\tau(pv-(z,0))=0$ and $\sigma(z,0)=0$
so finally
$$
pv=(pv-(z,0))+(z,0)\in\Inv_G(X)\oplus\Orth_G(X).
$$
This shows that $\frac{H^k(X,\IZ)}{\Inv_G^k(X)\oplus\Orth_G^k(X)}$
is a torsion module, and that it has only $p$-torsion.
\end{proof}

\begin{remark}\label{rem:H2STisGtrivial} As a consequence of the proof of Lemma~\ref{lem:H2STisptorsion}, observe that
 $\frac{H^k(X,\IZ)}{\Inv_G^k(X)\oplus\Orth_G^k(X)}$ is a trivial $G$-module: it is generated by the vectors
$v=(0,1)$ of each factor $(\mathcal{O}_K,a)$ appearing in the decomposition above, and the action of $G$ is 
$$
g\cdot v=a+v\equiv v\mod \Inv_G^k(X)\oplus\Orth_G^k(X).
$$
\end{remark}

\begin{definition}\label{def:aG} Assume that $H^*(X,\IZ)$ is torsion-free and $2\leq p\leq 19$. For $1\leq k \leq \dim_\IR X-1$ we define $\aG^k(X)\in\IN$ such that:
$$
\frac{H^k(X,\IZ)}{\Inv_G^k(X)\oplus\Orth_G^k(X)}\cong\left(\frac{\IZ}{p\IZ}\right)^{\aG^k(X)}.
$$
\end{definition}

\begin{lemma}\label{lem:fundamental_link}  Assume that $H^*(X,\IZ)$ is torsion-free.
For $1\leq k\leq\dim_\IR X-1$, one has: 
$$
\kappa_k\left(\Inv_G^k(X)\oplus \Orth_G^k(X)\right)=\ker(\bar\sigma)\cap H^k(X,\IF_p).
$$
\end{lemma}

\begin{proof}
It is clear that $\kappa_k(\Orth_G^k(X))\subset \ker(\bar\sigma)$. 
Similarly, since $\Inv_G^k(X)=\ker(\tau)$ and $\bar\sigma=\bar\tau^{p-1}$,
 one has $\kappa_k(\Inv_G^k(X))\subset\ker(\bar\sigma)$. Take $x\in H^k(X,\IZ)$ 
such that $\kappa(x)\in\ker(\bar\sigma)$. By Lemma~\ref{lem:H2STisptorsion} one can write 
$px=u+v$ for some $u\in\Inv_G^k(X)$, $v\in\Orth_G^k(X)$. 
Now $0=\bar\sigma\kappa(x)=\kappa\sigma(x)$ and $\sigma(x)=u$. This shows that $u=pu'$ 
for some $u'\in\Inv_G^k(X)$. 
Hence $px=pu'+v$, giving $v=pv'$ for some $v'\in\Orth_G^k(X)$, so finally $x\in\Inv_G^k(X)\oplus\Orth_G^k(X)$.
\end{proof}

\begin{corollary}\label{cor:linkZp}  Assume that $H^*(X,\IZ)$ is torsion-free. For $1\leq k\leq\dim_\IR X-1$ there is an isomorphism of $\IF_p$-vector spaces:
$$
\frac{H^k(X,\IZ)}{\Inv_G^k(X)\oplus\Orth_G^k(X)}\cong \frac{H^k(X,\IF_p)}{\ker(\bar\sigma)\cap H^k(X,\IF_p)}
$$
\end{corollary}

\begin{corollary}\label{cor:linkaGlp}  Assume that $H^*(X,\IZ)$ is torsion-free.
For $1\leq k\leq\dim_\IR X-1$ and $2\leq p\leq 19$ one has: 
$$
\aG^k(X)=\ell_p^k(X).
$$
\end{corollary}

\begin{proof}
By Corollary~\ref{cor:linkZp} one has
$\dim_{\IF_p}\ker\left(\bar\sigma_{|H^k(X,\IF_p)}\right)=h^k(X,\IF_p)-\aG^k(X)$ 
and by Lemma~\ref{lem:cohGadditif} and its proof, $\dim_{\IF_p}\ker\left(\bar\sigma_{|H^k(X,\IF_p)}\right)=h^k(X,\IF_p)-\ell_p^k(X)$.
The result follows.
\end{proof}

Assume now that $3\leq p\leq 19$. There is an exact sequence:
$$
0\lra (\sigma)\lra \IZ[G]\lra \IZ[\xi_p]\lra 0
$$
given by $g\mapsto\xi_p$. Since the $p$-th cyclotomic polynomial $\Phi_p(X)\in\IQ[X]$ is irreducible and $\sigma=\Phi_p(g)$, one deduces that $\Orth_G^k(X)$ is a free 
$\mathcal{O}_K$-module. Since $\mathcal{O}_K$ is a free $\IZ$-module of rank $p-1$, we introduce the following definition:

\begin{definition} \label{def:mG} Assume that $2\leq p\leq 19$. For $1\leq k\leq\dim_\IR X-1$ we define $\mG^k(X)\in\IN$ such that:
$$
\rank_\IZ\Orth_G^k(X)=\mG^k(X)(p-1).
$$
\end{definition}

\begin{corollary}\label{cor:linkmGlp}  Assume that $H^*(X,\IZ)$ is torsion-free.
For $1\leq k\leq\dim_\IR X-1$ and $3\leq p\leq 19$ one has:
$$
\mG^k(X)=\ell_p^k(X)+\ell_{p-1}^k(X).
$$
\end{corollary}

\begin{proof}
By Proposition~\ref{prop:decompCurtisReiner}(\ref{propCurtis:item4}) one has:
$$
\rank_\IZ\Inv_G^k(X)=\ell_1^k(X)+\ell_p^k(X)
$$
and by Lemma~\ref{lem:H2STisptorsion},
$\rank_\IZ\Inv_G^k(X)+\rank_\IZ\Orth_G^k(X)=h^k(X,\IZ)=h^k(X,\IF_p)$ one gets: 
$$
h^k(X,\IF_p)=\ell_1^k(X)+\ell_p^k(X)+\mG^k(X)(p-1).
$$
Since $h^k(X,\IF_p)=p\ell_p^k(X)+(p-1)\ell_{p-1}^k(X)+\ell_1^k(X)$, one gets the result.
\end{proof}

As a consequence of Corollary~\ref{cor:XGlq} and Proposition~\ref{prop:decompCurtisReiner}, using
Corollaries~\ref{cor:linkaGlp}~\&~\ref{cor:linkmGlp} we get the following explicit relation between the cohomology
of the fixed locus and the parameters $\aG,\mG$.

\begin{corollary}\label{cor:formulep=2} Assume that $p=2$, $H^\ast(X,\IZ)$ is torsion-free and the spectral sequence of equivariant cohomology with coefficients in $\IF_2$ degenerates at the $E_2$-term. Then:
$$
h^*(X^G,\IF_2)=h^\ast(X,\IF_2)-2\sum_{k=1}^{\dim_\IR X-1} \aG^k(X).
$$
\end{corollary}

\begin{corollary}\label{cor:formulepqqe}
Assume that $3\leq p\leq 19$, $H^\ast(X,\IZ)$ is torsion-free and the spectral sequence of equivariant cohomology with coefficients in $\IF_p$ degenerates at the $E_2$-term. Then:
$$
h^*(X^G,\IF_p)=h^\ast(X,\IF_p)-2\sum_{k=1}^{\dim_\IR X-1} \aG^k(X)-(p-2)\sum_{k=1}^{\dim_\IR X-1} \mG^k(X).
$$
\end{corollary}

\begin{proposition}
\label{prop:inequalities}
The following inequalities hold for all $k$:
\begin{align*}
0&\leq \mG^k(X)(p-1)\leq h^k(X,\IF_p),\\
0&\leq \aG^k(X)\leq\min\{\mG^k(X)(p-1),h^k(X,\IF_p)-\mG^k(X)(p-1)\}.
\end{align*}
\end{proposition}

\begin{proof}
One has $\frac{H^k(X,\IZ)}{\Inv_G^k(X)\oplus \Orth_G^k(X)}\cong\left(\frac{\IZ}{p\IZ}\right)^{\aG^k(X)}$: since $\Inv_G^k(X)$ and $\Orth_G^k(X)$ are primitive, the divisible classes are of the form $\frac{1}{p}(u+v)$ with $u\in\Inv_G^k(X)$ and $v\in\Orth_G^k(X)$. The integer $\aG^k(X)$ is the maximal number of divisible classes independant modulo $\Inv_G^k(X)\oplus\Orth_G^k(X)$ 
and is thus smaller than $\rank(\Inv_G^k(X))$ and $\rank(\Orth_G^k(X))$.
\end{proof}

%%%%%%%%%%%%%%%%%%%%%%%%%%%%%%%%%%%%%%%%%%%%%%%%%%%%%%%%%%%%%%%%%%
%%%%%%%%%%%%%%%%%%%%%%%%%%%%%%%%%%%%%%%%%%%%%%%%%%%%%%%%%%%%%%%%%%
%%%%%%%%%%%%%%%%%%%%%%%%%%%%%%%%%%%%%%%%%%%%%%%%%%%%%%%%%%%%%%%%%%

\section{Automorphisms of irreducible holomorphic symplectic manifolds}

\subsection{Basic facts on lattices}

Let $\left(\Lambda,\langle\cdot,\cdot\rangle\right)$ be a lattice (a free $\IZ$-module with an integral, bilinear symmetric, non-degenerate two-form). If $\Gamma\subset\Lambda$ is a sublattice, the \emph{dual lattice} is by definition $\Gamma^*\coloneqq \Hom_\IZ(\Gamma,\IZ)$. Recall that: 
$$
\Gamma^*\cong\{x\in\Gamma\otimes\IQ \,|\,\langle x,y\rangle\in \IZ\text{ for all } y\in\Gamma\}.
$$
Then $\Gamma\subset \Gamma^*$ is a sublattice of the same rank, so the quotient $A_\Gamma\coloneqq \Gamma^*/\Gamma$ is a finite abelian group, called the \emph{discriminant group}. Its order is denoted by $\disc(\Gamma)\coloneqq \left|A_\Gamma\right|$ and called the \emph{discriminant} of $\Gamma$. 

A lattice $\Lambda$ is called \emph{unimodular} if $\Lambda^*=\Lambda$, that is $A_\Lambda=0$. In a basis $(e_i)_i$ of~$\Lambda$, consider the matrix $M=(\langle e_i,e_j\rangle)_{i,j}$: then $\disc(\Lambda)=\det(M)$.

A sublattice $\Gamma\subset\Lambda$ is called \emph{primitive} if $\Lambda/\Gamma$ is a free $\IZ$-module.
If $\Lambda$ is unimodular and $\Gamma\subset\Lambda$ is primitive, then $A_\Gamma\cong A_{\Gamma^\perp}$.

Let $p$ be a prime number. A lattice $\Gamma$ is called $p$-\emph{elementary} if $A_\Gamma\cong(\IZ/p\IZ)^{a(\Gamma)}$ for some integer $a(\Gamma)$. In particular, $\disc(\Gamma)=p^{a(\Gamma)}$. If $\Gamma$ is primitively embedded in a unimodular lattice $\Lambda$, then $\Gamma^\perp$ is also $p$-elementary and $\disc(\Gamma)=\disc(\Gamma^\perp)$.

A lattice $\Gamma$ is called \emph{even} if $(x,x)\equiv 0\mod 2$ for all $x\in\Gamma$. Equivalently, the diagonal elements of the matrix $M$ are even.

\subsection{Basic facts on irreducible holomorphic symplectic manifolds}
\label{ss:basicsIHS}

A  compact K\"ahler manifold $X$ is called \emph{irreducible symplectic} if $X$ is simply connected and $H^0(X,\Omega_X^2)$ is spanned by an everywhere non-degenerate closed two-form, denoted $\omega_X$. We have a Hodge decomposition:
$$
H^2(X,\IC)=H^{2,0}(X)\oplus H^{1,1}(X)\oplus H^{0,2}(X)
$$
and we put $H^{1,1}(X)_\IR\coloneqq H^{1,1}(X)\cap H^2(X,\IR)$. The second cohomology group $H^2(X,\IZ)$ is torsion-free and equipped with a bilinear symmetric, even, non-degenerate two-form of signature $(3,b_2(X)-3)$, called the \emph{Beauville--Bogomolov} form \cite{Beauvillec1Nul}, and such that --- after scalar extension --- $H^{1,1}(X)$ is orthogonal to $H^{2,0}(X)\oplus H^{0,2}(X)$. We denote by $\langle\cdot,\cdot\rangle$ the bilinear form and by $q$ the associated quadratic form. The \emph{N\'eron-Severi} group of $X$ is defined by:
$$
\NS(X)\coloneqq H^{1,1}(X)_\IR\cap H^2(X,\IZ)=\left\{x\in H^2(X,\IZ)\,|\, \langle x,\omega_X\rangle=0\right\}.
$$
We set $\rho(X)\coloneqq \rank (\NS(X))$ the \emph{Picard number} of $X$ and $\Trans(X)\coloneqq \NS(X)^\perp$ the orthogonal complement of $\NS(X)$ in $H^2(X,\IZ)$ for the quadratic form, called the \emph{transcendental lattice}. Note that $\NS(X)$ and $\Trans(X)$ 
are primitively embedded in $H^2(X,\IZ)$.
We denote the signature of a lattice by $(n_1,n_2,n_3)$ where $n_1$ is the number of positive eigenvalues, $n_2$ of the zero eigenvalues and $n_3$ of the negative eigenvalues of the associated real quadratic form. There are three possibilities:
\begin{description}
\item[hyperbolic type] $\NS(X)$ is non--degenerate, of signature $(1,0,\rho(X)-1)$ and $\Trans(X)$ has signature $(2,0,b_2(X)-\rho(X)-2)$,
\item[parabolic type] $\NS(X)\cap\Trans(X)$ is of dimension $1$, $\NS(X)$ has signature $(0,1,\rho(X)-1)$ and $\Trans(X)$ has signature $(2,1,b_2(X)-\rho(X)-3)$,
\item[elliptic type] $\NS(X)$ is negative definite, of signature $(0,0,\rho(X))$ and $\Trans(X)$ has signature $(3,0,b_2(X)-\rho(X)-3)$.
\end{description} 
By Huybrechts \cite[Theorem 3.11]{Huybrechts}, $X$ is projective if and only if $\NS(X)$ is hyperbolic.

Let $G\subset\Aut(X)$ be a finite group of automorphisms of prime order $p$ and fix a generator $g\in G$. If $g^*\omega_X=\omega_X$ then $G$ is called \emph{symplectic}. Otherwise, there exists a primitive $p$-th root of the unity $\xi_p$ such that $g^*\omega_X=\xi_p\omega_X$ and $G$ is called \emph{non-symplectic}. For simplicity we put:
$$
\Inv_G(X)\coloneqq \Inv_G^2(X), \quad \Orth_G(X)\coloneqq \Orth_G^2(X), \quad \aG(X)\coloneqq \aG^2(X), \quad \mG(X)\coloneqq \mG^2(X).
$$

\begin{lemma}
If $p\leq 19$ then $\Orth_G(X)=\Inv_G(X)^\perp$ and $\Inv_G(X)$ is non-degenerate.
\end{lemma}

\begin{proof}
We already know that $\Inv_G(X)\cap\Orth_G(X)=\{0\}$ since $H^2(X,\IZ)$ is torsion-free. 
Take  ${x\in\Inv_G(X)}$ and ${y\in\Orth_G(X)}$. Using the $G$-invariance of 
the bilinear form one gets $\langle x,y\rangle=\langle x,g^iy\rangle$ for all $i$, hence $p\langle x,y\rangle=\langle x,\sigma(y)\rangle=0$ 
so $\Orth_G(X)\subset \Inv_G(X)^\perp$. By Lemma~\ref{lem:H2STisptorsion} one has $\rank \Inv_G(X)+\rank\Orth_G(X)=b_2(X)$.
Since the lattice $H^2(X,\IZ)$ is non-degenerate, we also have $\rank \Inv_G(X)+\rank\Inv_G(X)^\perp=b_2(X)$. It
follows that $\Orth_G(X)= \Inv_G(X)^\perp$ and  that the lattice $\Inv_G(X)$ is non-degenerate.
\end{proof}

Assume that $G$ is non-symplectic. Then $X$ is algebraic~\cite{BeauvilleKaehler}. 
If $1$ is an eigenvalue of $g$ on $\Trans(X)\otimes\IC$ and $t\in\Trans(X)$ is an eigenvector, one computes:
$$
\langle \omega_X,t\rangle=\langle g^*\omega_X,g^*t\rangle=\xi_p\langle\omega_X,t\rangle
$$
so $t\in\Trans(X)\cap\NS(X)=\{0\}$, contradiction. One deduces that the eigenvalues of the action of $g$ on $\Trans(X)$ are primitive $p$-th roots of 
the unity, so $\Trans(X)\subset\Orth_G(X)$. The minimal polynomial of $g$ on $\Trans(X)$ is the cyclotomic polynomial $\Phi_p$ 
hence ${\varphi(p)\leq b_2(X)-\rho(X)}$. 

Assume now that $G$ is symplectic and $X$ is algebraic or of elliptic type. If $t\in\Trans(X)$, one gets:
$$
\langle \omega_X,t\rangle=\langle\omega_X,g^*t\rangle
$$
so $g^*t-t\in\omega_X^\perp$. Since $X$ is algebraic, $\Trans(X)\cap\NS(X)=\{0\}$ so $g^*t=t$.
Then $g$ acts trivially on $\Trans(X)$, hence $\Orth_G(X)\subset \NS(X)$ and $\varphi(p)\leq \rho(X)$. This property remains true for the parabolic type if $X$ is isomorphic to the Hilbert scheme~$S^{[n]}$ of $n$ points on a K3 surface~$S$ with the following argument.
There exists an injective morphism $\iota\colon H^2(S,\IC)\to H^2(S^{[n]},\IC)$ such that $H^2(S^{[n]},\IZ)=\iota\left(H^2(S,\IZ)\right)\oplus\IZ \delta$,
where $2\delta$ is the class of the exceptional divisor of $S^{[n]}$. After normalisation, the form~$q$ satisfies $q(\iota(\alpha))=\alpha^2$ for $\alpha\in H^2(S,\IZ)$, $q(\delta)=-2(n-1)$ and $\delta$ is orthogonal to~$\iota\left(H^2(S,\IZ)\right)$ (see~\cite[Proposition~6]{BeauvilleKaehler}). Observe that $\NS(S^{[n]})=\iota\left(\NS(S)\right)\oplus\IZ\delta$ and $\Trans(S^{[n]})=\iota\left(\Trans(S)\right)$.
Setting $F\coloneqq \Trans(S^{[n]})\cap\NS(S^{[n]})\cong\IZ$, the previous argument shows that $G$ acts trivially on $\Trans(S^{[n]})/F$. Let $c$ be a generator of $F$. Then $c$ is of the form $c=\iota(c_0)$ with $c_0\in\Trans(S)\cap\NS(S)$ and $q(c)=c_0^2=0$ hence by the Riemann--Roch theorem, $c_0$ or $-c_0$ is an effective divisor. Changing $c$ to $-c$ if necessary, one can assume that $c_0$ is effective. 
Then $c$ is also an effective divisor hence $g^*c=c$. So $g^*$  acts trivially on $F$ and $\Trans(S^{[n]})/F$ hence all its eigenvalues on $\Trans(S^{[n]})$ are equal to one. Since $g$ is of finite order, $g^*$ is diagonalisable hence finally $g^*$ acts as the identity on $\Trans(S^{[n]})$.

\subsection{K3 surfaces} 

In this section, we assume that $X$ is a K3 surface. If $G$ is non-symplectic, then $\rho(X)\leq 21$ so $p\leq 19$. Nikulin~\cite{Nikulin} proved that 
for a symplectic action, one has in fact $p\leq 7$. The lattice $H^2(X,\IZ)$ is unimodular, hence the lattices $\Inv_G(X)$ and $\Orth_G(X)$ are $p$-elementary. 
Indeed, the generator $g$ of $G$ acts trivially on the discriminant groups $A_{\Inv_G(X)}\cong A_{\Orth_G(X)}$ but $\sigma$ is zero on $A_{\Orth_G(X)}$ hence for all $x\in A_{\Orth_G(X)}$ 
one has $px=0$, so these discriminant groups are $p$-torsion groups. It follows that the integer $\aG(X)$ has an important 
characterization: 
$$
\disc(\Inv_G(X))=\disc(\Orth_G(X))=p^{\aG(X)},
$$ 

so $\aG(X)=a\left(\Inv_G(X)\right)$.

\begin{corollary}
Let $X$ be a K3 surface and $G$ a group of automorphisms of prime order $p$. Then one has: 
$$
h^*(X^G,\IF_p)=\begin{cases}
24-2\aG(X) & \text{if } p=2 \text{ and } X^G\neq\emptyset,\\
24-(p-2)\,\mG(X)-2\aG(X) & \text{if } p\geq 3.
\end{cases}
$$
\end{corollary}

\begin{proof}
For $p=2$, if the group $G$ acts symplectically on $X$, by the holomorphic Lefschetz fixed point formula there are $8$ isolated fixed points. By Proposition~\ref{prop:degE2K3}, the spectral sequence of equivariant cohomology with coefficients in $\IF_p$ degenerates, so the formula is a consequence of Corollary~\ref{cor:XGlq}, Proposition~\ref{prop:decompCurtisReiner} and Corollary~\ref{cor:linkaGlp}. If $G$ acts non-symplectically and has fixed points, the same argument as above applies. Alternatively, one can observe that since $G$ acts locally at fixed points by quasi-reflections, by a result of Chevalley the quotient $X/G$ is smooth, so $H^\odd(X/G,\IF_p)=0$ and 
Proposition~\ref{prop:BorelSwan_p=2} (see below) gives the result. 

\par For $p\geq 3$, by the holomorphic Lefschetz fixed point formula the fixed locus is never empty so as above the spectral sequence degenerates and the result follows.
\end{proof}

\begin{remark}
This formula does not apply for $p=2$ when $X^G=\emptyset$: in this case, the quotient $X/G$ is 
an Enriques surface and it is well-known that $\aG(X)=10$.
\end{remark}

\begin{remark}
This formula due to Kharlamov~\cite{Kharlamov} for $p=2$ and Artebani--Sarti--Taki~\cite{AST} 
for $p\geq 3$. In both papers, it is obtained by using the classical theory of Smith sequences, 
under the assumption that $G$ acts non-symplectically. However, there is a small gap in the argument 
of Artebani--Sarti--Taki~\cite{AST}: if the fixed locus contains no isolated fixed points, 
then as above the quotient $X/G$ is smooth so $H^\odd(X/G,\IF_p)=0$ and as in~\cite{AST} 
Proposition~\ref{prop:BorelSwan_pqqe} below gives an equality. Otherwise the cohomology of the
 quotient may have $p$-torsion: for example, if the fixed locus contains only isolated fixed points, 
the long cohomology exact sequence of the pair $(X/G,X^G)$ gives $H^3(X/G,\IF_p)\cong\IF_p$ so 
Proposition~\ref{prop:BorelSwan_pqqe} gives only an inequality. Our argument using the degeneracy 
of the spectral sequence solves this problem.
\end{remark}

\subsection{Hilbert scheme of two points} 

Let $S$ be a K3 surface and assume that $X$ is deformation equivalent to the Hilbert scheme 
$S^{[2]}$ of two points on $S$. Then $1\leq \rho(X)\leq 21$ so $p\leq 23$.
Note that since $\rho(S^{[n]})\geq 2$, one has  $p\leq 19$ if $X=S^{[n]}$. If $X$ is deformation equivalent
to $S^{[n]}$ and $\rho(X)=1$ then the existence of non-symplectic automorphisms of order $p=23$ is not excluded.
For a symplectic action on a deformation of $S^{[2]}$, Mongardi~\cite{Mongardi2} shows that  $p\leq 11$ so higher
order automorphisms on deformations of $S^{[2]}$ are always non-symplectic. 

The following lemma is a direct generalisation of Nikulin's results~\cite{Nikulin}, 
very close from those stated in Mongardi~\cite{Mongardi1}. 

\begin{lemma}\label{lem:descriptionOrthInv} 
Assume that $X$ is deformation equivalent to $S^{[2]}$ and that $G$ is an order~$p$ group of automorphisms
 of $X$ with $3\leq p\leq 19$. Then the lattice~$\Orth_G(X)$ has discriminant $A_{\Orth_G(X)}\cong\left(\frac{\IZ}{p\IZ}\right)^{\aG(X)}$ 
and in the symplectic case contains no $(-2)$-classes. The invariant lattice $\Inv_G(X)$ has  discriminant
 $A_{\Inv_G(X)}\cong\left(\frac{\IZ}{2\IZ}\right)\oplus\left(\frac{\IZ}{p\IZ}\right)^{\aG(X)}$. 
Moreover, if $G$ acts symplectically then $\Orth_G(X)$ is negative definite of rank $(p-1)\mG(X)$ 
and $\Inv_G(X)$ has signature $(3,20-(p-1)\mG(X))$. If $G$ acts non-symplectically, then $\Orth_G(X)$ 
has signature $(2,(p-1)\mG(X)-2)$ and $\Inv_G(X)$ has signature $(1,22-(p-1)\mG(X))$.
\end{lemma}

\begin{proof}
For the fact $\Orth_G(X)$ contains no $(-2)$ classes when the action is symplectic, see Mongardi~\cite{Mongardi1}.
By Lemma~\ref{lem:H2STisptorsion}, from the relation
$$
[H^2(X,\IZ):\Inv_G(X)\oplus\Orth_G(X)]^2=\disc(\Inv_G(X))\cdot \disc(\Orth_G(X))\cdot \disc(H^2(X,\IZ))^{-1}
$$
 we get the formula $\disc(\Inv_G(X))\cdot \disc(\Orth_G(X))=2\,p^{2\aG(X)}$, hence $\disc(\Orth_G(X))=2^{\epsilon}p^{\alpha}$ and $\disc(\Inv_G(X))=2^{1-\epsilon}p^{\beta}$
with $\epsilon\in\{0,1\}$ since $p$ is odd, with $\alpha+\beta=2\aG(X)$. 
Since the inclusion $\Inv_G(X)\subset H^2(X,\IZ)$ is primitive, as explained in 
Nikulin~\cite[\S 5]{Nikulinfactor} the inclusion 
$$
M:=\frac{H^2(X,\IZ)}{\Inv_G(X)\oplus \Orth_G(X)}\subset A_{\Inv_G(X)}\oplus A_{\Orth_G(X)}
$$
is such that the projections $p\colon M\to A_{\Inv_G(X)}$ and $q\colon M\to A_{\Orth_G(X)}$ 
are $G$-equivariant monomorphisms. We deduce that $\aG(X)\leq\alpha$ and $\aG(X)\leq\beta$. 
This shows that $\alpha=\beta=\aG(X)$.

 We show now that  $G$ acts trivially on $A_{\Orth_G(X)}$. 
There are two possibilities: 
\begin{enumerate}
\item\label{itemi} $M\cong  A_{\Inv_G(X)}$ and $A_{\Orth_G(X)}/M\cong\IZ/2\IZ$,
\item\label{itemii} $M\cong  A_{\Orth_G(X)}$ and $ A_{\Inv_G(X)}/M\cong\IZ/2\IZ$. 
\end{enumerate}
By Remark~\ref{rem:H2STisGtrivial}, $M\cong \left(\frac{\IZ}{p\IZ}\right)^{\aG(X)}$ is a trivial 
$G$-module so in case (\ref{itemii}) the result is clear.
In case (\ref{itemi}) one has a $G$-equivariant inclusion
$$
M=\left(\frac{\IZ}{p\IZ}\right)^{\aG(X)}\to\left(\frac{\IZ}{2\IZ}\right)\oplus\left(\frac{\IZ}{p\IZ}\right)^{\aG(X)}=A_{\Orth_G(X)}.
$$                                                                                                                 
Since $p$ is odd, this map is trivial on the first factor. Since $M$ is a trivial 
$G$-module this shows that $G$ acts trivially on $A_{\Orth_G(X)}$. 

Since $\Orth_G(X)=\Ker(\sigma)$ and $G$ acts trivially on $A_{\Orth_G(X)}$ it follows that $\Orth_G(X)$ 
is $p$-elementary so $\epsilon=0$. This shows that case (\ref{itemi}) cannot occur so we have $M\cong A_{\Orth_G(X)}\cong \left(\frac{\IZ}{p\IZ}\right)^{\aG(X)}$
and $A_{\Inv_G(X)}\cong\left(\frac{\IZ}{2\IZ}\right)\oplus\left(\frac{\IZ}{p\IZ}\right)^{\aG(X)}$. 

If $G$ acts symplectically, the invariant lattice $\Inv_G(X)\otimes_\IZ\IC$ contains the symplectic form $\omega_X$, 
its conjugate $\overline{\omega_X}$ and an invariant K\"ahler 
class. Since $H^2(X,\IZ)$ has signature $(3,20)$, this implies that $\Inv_G(X)$ has signature $(3,20-(p-1)\mG(X))$ and $\Orth_G(X)$ has signature $(0,(p-1)\mG(X))$.
If $G$ acts non symplectically, then $\omega_X,\overline{\omega_X}\in\Orth_G(X)\otimes\IC$ and $\Inv_G(X)$ still contains an invariant K\"ahler class, hence $\Orth_G(X)$ has signature $(2,(p-1)\mG(X)-2)$ and $\Inv_G(X)$ has signature $(1,22-(p-1)\mG(X))$.
\end{proof}

Markman~\cite{Markman} proved that 
$H^\text{odd}(X,\IZ)=0$ and $H^\text{even}(X,\IZ)$ is torsion-free, 
and Verbitsky~\cite{Verbitsky} proved that the cup product map 
${\Sym^2 H^2(X,\IQ)\to H^4(X,\IQ)}$ is an isomorphism. 
We first study  the embedding ${\Sym^2H^2(X,\IZ)\hookrightarrow H^4(X,\IZ)}$.

\begin{proposition} \label{prop:indiceS2H4}
If $X$ is deformation equivalent to $S^{[2]}$, then: $$
\frac{H^4(X,\IZ)}{\Sym^2H^2(X,\IZ)}\cong\left(\frac{\IZ}{2\IZ}\right)^{\oplus 23}\oplus\left(\frac{\IZ}{5\IZ}\right).
$$ 
\end{proposition}

\begin{proof}
It is enough to do the computation for $X=S^{[2]}$. Following an observation of O'Grady~\cite{OGrady}, we define on $\Sym^2H^2(X,\IZ)$ a bilinear symmetric pairing by
$$
\ll \alpha_1\odot\alpha_2,\alpha_3\odot\alpha_4\gg:=\langle\alpha_1,\alpha_2\rangle\langle\alpha_3,\alpha_4\rangle+\langle\alpha_1,\alpha_3\rangle\langle\alpha_2,\alpha_4\rangle+\langle\alpha_1,\alpha_4\rangle\langle \alpha_2,\alpha_3\rangle
$$
for $\alpha_1,\alpha_2,\alpha_3\alpha_4\in H^2(X,\IZ)$. Note that for any $\alpha\in H^2(X,\IZ)$ one has 
$$
\ll\alpha\odot\alpha,\alpha\odot\alpha\gg=3\langle\alpha,\alpha\rangle=\int_{X}\alpha^4.
$$
It follows that $(\Sym^2H^2(X,\IZ),\ll\cdot,\cdot\gg)$ is a sublattice of $H^4(X,\IZ)$ equipped with the Poincar\'e pairing. Since this lattice is unimodular, one has
$$
[H^4(X,\IZ),\Sym^2 H^2(X,\IZ)]^2=\disc(\Sym^2 H^2(X,\IZ)).
$$
Using that $H^2(X,\IZ)$ is isometric to $U^{\oplus 3}\oplus E_8(-1)^{\oplus 2}\oplus\langle -2\rangle$ it is easy to compute that the $\Sym^2 H^2(X,\IZ)$ has discriminant $\disc(\Sym^2 H^2(X,\IZ))=2^{46}\cdot 5^2$.  The result follows.
\end{proof}

\begin{remark} In order to understand the (somehow surprising) $5$-torsion class, one can perform an explicit computation as follows. For $\alpha\in H^*(S,\IZ)$ and $i\in\IZ$, we denote by $\kq_i(\alpha)\in\End(H^*(X,\IZ))$ the Nakajima operators~\cite{Nakajima} and by $\vac\in H^0(S^{[0]},\IZ)$ the unit. Let $(\alpha_i)_{i=1,\ldots,22}$ be an integral basis of $H^2(S,\IZ)$, denote by $1\in H^0(S,\IZ)$ the unit and by $x\in H^4(S,\IZ)$ the class of a point. The results of Qin--Wang~\cite[Theorem 5.4, Remark 5.6]{QinWang} give the following integral basis:
\begin{itemize}
\item integral basis of $H^2(X,\IZ)$: $\frac{1}{2}\kq_2(1)\vac$, $\kq_1(1)\kq_1(\alpha_i)\vac$,
\item integral basis of $H^4(X,\IZ)$: 
\begin{align*}
\kq_1(1)\kq_1(x)\vac,\quad \kq_2(\alpha_i)\vac,\quad\kq_1(\alpha_i)\kq_1(\alpha_j)\vac \text{ with } i<j, \\
\km_{1,1}(\alpha_i)\vac=\frac{1}{2}\left(\kq_1(\alpha_i)^2-\kq_2(\alpha_i)\right)\vac.
\end{align*}
\end{itemize}
The cup product map $\Sym^2 H^2(X,\IQ)\to H^4(X,\IQ)$ can be computed explicitly by using the algebraic model constructed by 
Lehn--Sorger~\cite{LehnSorger}: 

\par$\bullet$ for $\alpha\in H^2(S,\IZ)$: $\frac{1}{2}\kq_2(1)\vac\cup \kq_1(1)\kq_1(\alpha)\vac=\kq_2(\alpha)\vac$,
\par$\bullet$ for $\alpha,\beta\in H^2(S,\IZ)$:
$$
\kq_1(1)\kq_1(\alpha)\vac\cup\kq_1(1)\kq_1(\beta)\vac=\left(\int_S \alpha\beta\right)\kq_1(1)\kq_1(x)\vac+\kq_1(\alpha)\kq_1(\beta)\vac,
$$
\par$\bullet$ denote by  $\delta\colon S\to S\times S$ the diagonal embedding. We denote  the push-forward map followed by the K\"unneth isomorphism by ${\delta_*\colon H^*(S,\IZ)\to H^*(S,\IZ)\otimes H^*(S,\IZ)}$. Writing
$\delta_* 1=\sum_{i,j} \mu_{i,j}\alpha_i\otimes\alpha_j+1\otimes x+x\otimes 1$ for some $\mu_{i,j}\in\IZ$ with the property that $\mu_{i,j}=\mu_{j,i}$, one has:
$$
\frac{1}{2}\kq_2(1)\vac\cup\frac{1}{2}\kq_2(1)\vac=\sum_{i<j}\mu_{i,j}\kq_1(\alpha_i)\kq_1(\alpha_j)\vac+\frac{1}{2}\sum_i \mu_{i,i}\kq_1(\alpha_i)^2\vac+\kq_1(1)\kq_1(x)\vac.
$$

An elementary computation (with the help of a computer) allows to determine the $253$ coefficients $\mu_{i,j}$, by using the fact that $\delta_*$ is the adjoint of the cup-product and the intersection matrix of $H^2(S,\IZ)\cong U^{\oplus 3}\oplus E_8(-1)^{\oplus 2}$ (where $U$ denotes the rank $2$ hyperbolic lattice). One can then express the basis of $\Sym^2H^2(X,\IZ)$ in terms of the given integral basis of $H^4(X,\IZ)$ and compute the Smith normal form of the quotient $\frac{H^4(X,\IZ)}{\Sym^2H^2(X,\IZ)}$. One finds $\left(\frac{\IZ}{2\IZ}\right)^{\oplus 22}\oplus\left(\frac{\IZ}{10\IZ}\right)$. 

A precise look at the computation of the Smith normal form shows that the $2$-divisible classes in $\Sym^2 H^2(X,\IZ)$ are the six vectors $\km_{1,1}(\alpha_i)\vac$ for the basis elements $\alpha_i\in U^{\oplus 3}$ and the $16$ vectors $\km_{1,1}(\alpha_i)\vac-\kq_1(1)\kq_1(x)\vac$ for the basis elements $\alpha_i\in E_8(-1)^{\oplus 2}$. Finally the $10$-divisible class 
is $\kq_1(1)\kq_1(x)\vac$. 
\end{remark}

\begin{remark} 
If $X$ is deformation equivalent to $S^{[n]}$ with $n\geq 4$, Markman~\cite[Theorem 1.10]{Markman2} (see also \cite[Theorem 9.3]{MarkmanTorelli}) proved that the quotient $\frac{H^4(S^{[n]},\IZ)}{\Sym^2H^2(S^{[n]},\IZ)}$ is 
free of rank $24$. The case $n=3$ remains to be computed.
\end{remark}

Let $S$ be a K3 surface and $G$ a finite group of automorphisms of prime order $p$ on $S^{[2]}$. 
We give some degeneracy results of the spectral sequence of equivariant cohomology with coefficients in $\IF_p$.

\begin{lemma}\label{lem:DeligneForE} Denote by $e\coloneqq 2\delta$ the class of the exceptional divisor of $S^{[2]}$. Then $\int_{S^{[2]}}e^4=2^6\cdot 3$
and the map $H^2(S^{[2]},\IZ)\to H^6(S^{[2]},\IZ),\beta\mapsto e^2\beta$ has discriminant~$2^{70}\cdot 3$.
\end{lemma}

\begin{proof}
By Beauville \cite{Beauvillec1Nul} and Fujiki \cite{Fujiki} there exists a constant $c_{S^{[2]}}$ such that 
$$
\int_{S^{[2]}}e^4=c_{S^{[2]}}\,q(e)^2. 
$$
One has $q(e)=-8$ and by Markushevich \cite[Proposition~1.2]{Markushevich} we have $c_{S^{[2]}}=3$, so $\int_{S^{[2]}}e^4=2^6\cdot 3$. 
The multiplication map $H^2(S^{[2]},\IZ)\to H^6(S^{[2]},\IZ),\beta\mapsto e^2\beta$ is 
equivalent by Poincar\'e duality to the bilinear form $H^2(S^{[2]},\IZ)\times H^2(S^{[2]},\IZ)\to\IZ$,
 $(\alpha,\beta)\mapsto \int_{S^{[2]}}\alpha e^2\beta$.
From the formal relation in $\IZ[x,y,z]$
$$
\int_{S^{[2]}}(\alpha x+e y+\beta z)^4=3\,q(\alpha x+e y+\beta z)^2,
$$
and by extracting the coefficient of $xy^2z$ we get the formula
$$
\int_{S^{[2]}} \alpha e^2 \beta=-8\,\langle\alpha,\beta\rangle+2\,\langle\alpha,e\rangle\, \langle e,\beta\rangle. 
$$
Denoting by $\Lambda\coloneqq E_8(-1)^{\oplus 2}\oplus U^{\oplus 3}$ the K3 lattice (the lattice $H^2(S,\IZ)$ for the intersection product),
the lattice $H^2(S^{[2]},\IZ)\cong H^2(S,\IZ)\oplus\IZ\delta$ equiped with this bilinear form is then isometric to 
$\Lambda(-8)\oplus \langle 48\rangle$, so its discriminant is $2^{70}\cdot 3$.
\end{proof}

\begin{lemma}\label{lem:DeligneForc2} Denote by $c_2(S^{[2]})$ the second Chern class of the tangent bundle on~$S^{[2]}$. Then
$\int_{S^{[2]}}c_2(S^{[2]})^2=2^2\cdot 3^2\cdot 23$ and the multiplication map 
$$
H^2(S^{[2]},\IZ)\to H^6(S^{[2]},\IZ),\beta\mapsto c_2(S^{[2]})\beta
$$ 
has discriminant~$^{24}\cdot 3^{23}\cdot 5^{23}$.
\end{lemma}

\begin{proof}
The Chern number $\int_{S^{[2]}}c_2(S^{[2]})^2=828=2^2\cdot 3^2\cdot 23$ is well-known (see Ellingsrud--G\"ottsche-Lehn~\cite{EGL} 
or Nieper-Wisskirchen~\cite[Remark~4.13]{NW}). To compute the discriminant of the multiplication map, similarly as in the proof of the
previous lemma we consider the quadratic form $H^2(S^{[2]},\IZ)\to\IZ, \alpha\mapsto c_2(S^{[2]})\alpha^2$. By Nieper-Wisskirchen~\cite[Corollary~3.8]{NW} one has
\begin{align*}
\int_{S^{[2]}} c_2(S^{[2]})\alpha^2&=96\,\lambda(\alpha)\td^{\frac{1}{2}}(S^{[2]}),\\
\int_{S^{[2]}}\alpha^4&=24\lambda(\alpha)^2\td^{\frac{1}{2}}(S^{[2]}).
\end{align*}
Recall that $\int_{S^{[2]}}\alpha^4=3\,q(\alpha)^2$ and that the square root of the Todd genus is here:
$$
\sqrt{\td}(S^{[2]})=1+\frac{1}{24}c_2(S^{[2]})+\frac{7}{5760}c_2(S^{[2]})^2-\frac{1}{1440}c_4(S^{[2]}).
$$ 
Using that $c_4(S^{[2]})=324$ (see~\cite{NW}) one gets 
$\td^{\frac{1}{2}}(S^{[2]})\coloneqq\int_{S^{[2]}}\sqrt{\td}(S^{[2]})=\frac{25}{32}$.
Putting all together we get
$$
\left(\int_{S^{[2]}} c_2(S^{[2]})\alpha^2\right)^2=900\,q(\alpha)^2
$$
so $\int_{S^{[2]}} c_2(S^{[2]})\alpha^2=\pm 30q(\alpha)$. For this quadratic form, the lattice $H^2(S^{[2]},\IZ)$ is then isometric
to $\Lambda(30)\oplus\langle -60\rangle$, so the multiplication map has discriminant
 $2^{24}\cdot 3^{23}\cdot 5^{23}$.
\end{proof}

Let $(e_i)_i$ be an orthonormal basis of $H^2(X,\IC)$ for $q$. The image of the class $\sum_i e_i\cdot e_i\in\Sym^2H^2(X,\IC)$ in $H^4(X,\IC)$ is denoted $q^{-1}$ and called
the \emph{Beauville--Bogomolov class}. By Markman~\cite{Markman3}, for $X$ deformation equivalent to $S^{[n]}$ there is a decomposition
$$
q^{-1}=c_2(X)+2\kappa_2(E_x)
$$
where $x\in X$ is a point, $E_x$ a rank $2n-2$ reflexive coherent twisted sheaf and $\kappa(E_x)\coloneqq ch(E_x)\exp\left(\frac{-c_1(E_x)}{2n-2}\right)$. Since $\kappa_2=\frac{c_1^2}{8}-c_2$, 
the class $u\coloneqq 4q^{-1}$ lives in $H^4(S^{[2]},\IZ)$. As noted by O'Grady~\cite{OGrady} and Markman~\cite{Markman3},
for $n=2$ the classes $q^{-1}$, $c_2(X)$ and $\kappa_2(E_x)$ span a $1$-dimensional space
and in fact $q^{-1}=\frac{5}{6}c_2(X)$ in $H^2(X,\IQ)$.

\begin{lemma}\label{lem:DeligneForq} Assume that $X$ is deformation equivalent to $S^{[2]}$. One has $\int_X u^2=2^4\cdot 5^2\cdot 23$ and the map $H^2(X,\IZ)\to H^6(X,\IZ)$, $\beta\mapsto u\beta$ has 
discriminant $2^{47}\cdot 5^{46}$.
\end{lemma}

\begin{proof} This is a direct consequence of the equality $3u=10 c_2(X)$ in $H^2(X,\IZ)$. Here is an alternative, more direct argument. Let $(e_i)_i$ be an orthonormal basis of $H^2(X,\IC)$ for $q$, $v\coloneqq\sum_i x_i e_i\in H^2(X,\IC)$ and $\alpha\in H^2(X,\IC)$. From the relation
$$
\int_X(v+\alpha)^4=3q(v+\alpha)^2
$$
and by extracting the quadratic part in $\alpha$ we get
$$
\int_X v^2\alpha^2=q(v)q(\alpha)+2\langle v,\alpha\rangle^2.
$$
By extracting the square coefficients in the variables $x_i$ and putting them to one we get
\begin{align*}
\int_X\left(\sum_i e_i^2\right)\alpha^2&=q\left(\sum_i e_i\right)q(\alpha)+2\sum_i\langle e_i,\alpha\rangle^2\\
\int_X q^{-1}\alpha^2&=b_2(X)q(\alpha)+2q(\alpha)\\
&=25q(\alpha)
\end{align*}
Since $u=4q^{-1}$ this implies that $\int_X u\alpha^2=2^2 5^2 q(\alpha)$. For this quadratic form, the lattice $H^2(X,\IZ)$ is then isometric to $\Lambda(2^25^2)\oplus\langle -2^35^2\rangle$ so the multiplication map has 
discriminant $2^{47}\cdot 5^{46}$.
Taking $\alpha=\sum_i x_i e_i$, we also have
$$
\int_X u \left(\sum_i x_ie_i\right)^2=2^25^2 q\left(\sum_i x_i e_i\right).
$$
By extracting the square coefficients in $x_i$ we get similarly
$$
\int_X u\sum_i e_i^2=2^25^2\sum_i q(e_i)
$$
hence $\int_X u^2=2^4\cdot 5^2\cdot 23$.

\end{proof}

\begin{proposition}\label{prop:degenere} Let $G$ be an order $p$ group acting on $X$. The spectral sequence of equivariant cohomology with coefficients in $\IF_p$ degenerates at the $E_2$-term in the following cases:
\begin{enumerate}
\item\label{propdegenere:item1} $X$ is deformation equivalent to $S^{[2]}$, $3\leq p\leq 19$ and $p\neq 5$.
\item\label{propdegenere:item2} $G$ acts by natural automorphisms on $S^{[2]}$ and $p\neq 2$.
\end{enumerate}
\end{proposition}

\begin{proof} \text{}\\
\noindent{(\ref{propdegenere:item1})} If $X$ is deformation equivalent to $S^{[2]}$, the class $c_2(T_X)$ is $G$-equivariant and can be used in Proposition~\ref{prop:degE2Hilb2}. By lemma~\ref{lem:DeligneForc2} the degeneracy conditions are fullfilled for $p>5$. For $p=3$, by Lemma~\ref{lem:DeligneForq} one can use the class $u=4q^{-1}$ since from the definition of $q^{-1}$ it is clear that this class is $G$-invariant (alternatively, $u$ is proportional to $c_2(X)$ in $H^2(X,\IQ)$ so it is $G$-equivariant).

\noindent{(\ref{propdegenere:item2})} If $G$ acts by natural automorphisms on $S^{[2]}$, the exceptional divisor is $G$-invariant and it can be used in Proposition~\ref{prop:degE2Hilb2}. By Lemma \ref{lem:DeligneForE} the degeneracy conditions are fullfilled if $p>3$. In the case $p=3$ the class $u=4q^{-1}$ can be used again.
\end{proof}

\begin{lemma} Assume that $G=\IZ/2\IZ$ and let $M$ be a finite-dimensional $\IF_2[G]$-module. Then:
\begin{align*}
\ell_1(\Sym^2M)&=\frac{\ell_1(M)(\ell_1(M)+1)}{2}+\ell_2(M),\\
\ell_2(\Sym^2M)&=\ell_2(M)(\ell_2(M)+\ell_1(M)).
\end{align*}
\end{lemma}

\begin{proof}
The $\IF_2[G]$-modules $M$ and $\Sym^2M$ decompose as:
\begin{align*}
M&\cong N_2^{\oplus \ell_2(M)}\oplus N_1^{\oplus \ell_1(M)},\\
\Sym^2M&\cong N_2^{\oplus \ell_2(\Sym^2M)}\oplus N_1^{\oplus \ell_1(\Sym^2M)}.
\end{align*}
By elementary matrix computations, one finds the following $G$-module decompositions:
\begin{align*}
\Sym^2N_2&\cong N_2\oplus N_1, & N_2\otimes N_2&\cong N_2^{\oplus 2},& N_2\otimes N_1&\cong N_2\\
\Sym^2N_1&\cong N_1, & N_1\otimes N_1&\cong N_1.
\end{align*}
The result follows.
\end{proof}

\begin{lemma}\label{lem:lpSym2} Assume that $3\leq p\leq 19$, $G=\IZ/p\IZ$ and let $M$ be a finite-dimensional $\IF_p[G]$-module. Then:
\begin{align*}
\ell_1(\Sym^2M)&=\frac{\ell_1(M)\cdot(\ell_1(M)+1)}{2}+\frac{\ell_{p-1}(M)\cdot(\ell_{p-1}(M)-1)}{2},\\
\ell_{p-1}(\Sym^2M)&=\ell_{p-1}(M)\cdot\ell_1(M),\\
\ell_p(\Sym^2M)&=\frac{p+1}{2}\cdot\ell_p(M)+p\cdot\frac{\ell_p(M)\cdot(\ell_p(M)-1)}{2}+\frac{p-1}{2}\cdot\ell_{p-1}(M)\\
&\quad+(p-1)\cdot\ell_p(M)\cdot\ell_{p-1}(M)+\ell_p(M)\cdot\ell_1(M)\\
&\quad+(p-2)\cdot\frac{\ell_{p-1}(M)\cdot(\ell_{p-1}(M)-1)}{2},
\end{align*}
and $\ell_i(\Sym^2M)=0$ for $2\leq i\leq p-2$.
\end{lemma}

\begin{proof}
As before, we have the decompositions:
\begin{align*}
M&\cong N_p^{\oplus \ell_p(M)}\oplus N_{p-1}^{\oplus\ell_{p-1}(M)}\oplus N_1^{\oplus\ell_1(M)},\\
\Sym^2M&\cong \bigoplus_{1\leq q\leq p}N_q^{\oplus \ell_q(\Sym^2M)}.
\end{align*}
By elementary matrix computations, one finds the following $G$-module decompositions:
\begin{align*}
\Sym^2 N_p&\cong N_p^{\oplus \frac{p+1}{2}}, & N_p\otimes N_p&\cong N_p^{\oplus p},& N_{p-1}\otimes N_{p-1}&\cong N_p^{\oplus p-2}\oplus N_1\\
\Sym^2 N_{p-1}&\cong N_p^{\oplus \frac{p-1}{2}}, & N_p\otimes N_{p-1}&\cong N_p^{\oplus p-1}, & N_{p-1}\otimes N_1&\cong N_{p-1}\\
\Sym^2 N_1&\cong N_1, & N_p\otimes N_1&\cong N_p, & N_1\otimes N_1&\cong N_1.
\end{align*}
The result follows.
\end{proof}

\begin{corollary}\label{cor:mainformula}
Let $X$ be deformation equivalent to $S^{[2]}$ and $G$ be a group of automorphisms of prime order $p$ on $X$ with
$3\leq p\leq 19$ and $p\neq 5$. Then:
\begin{align*}
h^\ast(X^G,\IF_p)=&324-2\aG(X)\left(25-\aG(X)\right)-(p-2)\mG(X)\left(25-2\,\aG(X)\right)\\
&+\frac{1}{2}\mG(X)\left((p-2)^2\mG(X)-p\right)
\end{align*}
with 
\begin{align*}
2&\leq (p-1)\mG(X)<23,\\
0&\leq \aG(X)\leq \min\{(p-1)\mG(X),23-(p-1)\mG(X)\}.
\end{align*}
\end{corollary}

\begin{proof} By Proposition~\ref{prop:degenere} and Corollary~\ref{cor:formulepqqe}, using Poincar\'e duality
one has:
$$
h^\ast(X^G,\IF_p)=324-4\aG(X)-2\aG^4(X)-2(p-2)\mG(X)-(p-2)\mG^4(X).
$$
By Proposition~\ref{prop:indiceS2H4} one has an isomorphism $\Sym^2 H^2(X,\IF_p)\cong H^4(X,\IF_p)$ so by Lemma~\ref{lem:lpSym2}, using Corollaries~\ref{cor:linkaGlp}~\&~\ref{cor:linkmGlp} one can express the parameters $\aG^4(X)$ and $\mG^4(X)$ in terms of $\aG(X)$ and $\mG(X)$.
One gets easily the formula.
The estimates for the parameters $\aG(X),\mG(X)$ come from Proposition~\ref{prop:inequalities}, noting 
that $p$ is odd and that $\mG(X)$ cannot be zero, otherwise $\Inv_G(X)=H^2(X,\IZ)$ so $G$ acts trivially on
$H^2(X,\IZ)$. By Beauville~\cite[Proposition~10]{BeauvilleKaehler} this is impossible since $G\neq\{\id\}$.
\end{proof}

\begin{remark} \text{}
\begin{enumerate}
\item If $G$ acts symplectically, the inclusions $\Trans(X)\subset \Inv_G(X)$ and {$\Orth_G(X)\subset \NS(X)$} give one more relation: $(p-1)\mG(X)\leq \rho(X)$.
 If instead $G$ acts non symplectically, the inclusions $\Trans(X)\subset \Orth_G(X)$ and $\Inv_G(X)\subset \NS(X)$ give $23-(p-1)\mG(X)\leq\rho(X)$.
\item Assume that $G$ is an order $p$ group of automorphisms of $S$. We denote also by $G$ 
the group of natural automorphisms induced on $S^{[2]}$. Since the exceptional
divisor of $S^{[2]}$ is invariant by $G$, it is clear that
 $\aG(S^{[2]})=\aG(S)$ and $\mG(S^{[2]})=\mG(S)$.
 For example, if $G$ acts symplectically and $p=3$, by Nikulin~\cite{Nikulin} $G$ 
has $6$ isolated fixed points on $S$ and by Garbagnati--Sarti~\cite[Theorem~4.1]{GS} 
one has $\aG(S)=6$ and $\mG(S)=6$. By Corollary~\ref{cor:mainformula} 
we get $h^\ast(S^{[2]},\IF_5)=27$. In this case it is easy to see that the fixed 
locus $(S^{[2]})^G$ consists indeed in $27$ isolated points \cite[Exemple~2]{Boissiere}.
\item If $X$ is deformation equivalent to $S^{[2]}$ and $p\in\{17,19\}$, it follows from the inequalities
given in Corollary~\ref{cor:mainformula} that $\mG(X)=1$. Then necessarily $\aG(X)>0$. Otherwise,
by Lemma~\ref{lem:descriptionOrthInv} the lattice $\Orth_G(X)$ 
would be even unimodular of signature $(2,14)$ for $p=19$ or $(2,14)$ for $p=17$. By an theorem of Milnor
(see Nikulin~\cite[Theorem~1.1.1]{Nikulinfactor}) such lattices do not exist.
\end{enumerate}
\end{remark}

\subsection{Applications}
\label{ss:applications}

\subsubsection{Existence of fixed points}

\begin{proposition} Let $X$ be deformation equivalent to $S^{[2]}$ and $G$ be a group of automorphisms of prime order on $X$. Then the fixed locus $X^G$ is not empty.
\end{proposition}

\begin{proof}
Denote by $g$ a generator of $G$ and by $\xi_p$ a primitive 
$p$-th root of the unity. If $g$ acts symplectically, its holomorphic Lefschetz number for 
the sheaf $\cO_X$  is three so $G$ has fixed points. If $g$ acts non-symplectically, 
its holomorphic Lefschetz number is $1+\xi_p+\xi_p^2$, so $G$ has fixed points if $p\neq 3$. If $p=3$ one can use Corollary~\ref{cor:mainformula} to check all possible values of $h^\ast(X^G,\IF_3)$. One finds that it is never zero, so $X^G$ is never empty.
\end{proof}

\begin{corollary} 
Let $X$ be deformation equivalent to $S^{[2]}$ and $G$ be a finite group of automorphisms of $X$. Then $G$ does not act freely on $X$.
\end{corollary}

\begin{remark}
This result implies that it is not possible to construct Enriques varieties of dimension four and index three, as defined in Boissi\`ere--Nieper-Wi{\ss}kirchen--Sarti~\cite{BNWS} and Oguiso--Schro\"er~\cite{OS}, if one starts with deformations of Hilbert schemes of two points on a K3 surface.
\end{remark}

\subsubsection{An automorphism of order eleven} 

Consider the cubic $C$ in $\IP^5$ given by the equation
$$
x_0^3+x_1^2x_5+x_2^2x_4+x_3^2x_2+x_4^2x_1+x_5^2x_3=0.
$$
Let $\xi$ be a primitive eleventh root of the unit and consider the order $11$ automorphis~$\varphi$ of~$C$ given by 
$$
\varphi(x_0,x_1,x_2,x_3,x_4,x_5)=(x_0,\xi x_1,\xi^3x_2,\xi^4x_3,\xi^5x_4,\xi^9 x_5).
$$
As explained in Mongardi~\cite{Mongardi2}, $\varphi$ induces a symplectic automorphism of order $11$
on the Fano variety of lines $X$ of the cubic fourfold $C$, with $5$ isolated fixed
 points.
Using our main formula  given in Corollary~\ref{cor:mainformula}, one finds that there is only one possibility
for the parameters $\aG(X),\mG(X)$, that is:
$$
\aG(X)=2,\quad  \mG(X)=2,
$$ 
hence $\rho(X)\in\{20,21\}$. Since $X$ is algebraic, the inclusion $\Orth_G(X)\subset \NS(X)$ is strict since the N\'eron-Severi lattice is hyperbolic, so $\rho(X)=21$ (see also Mongardi~\cite{Mongardi2} for a more general statement). It follows from Lemma~\ref{lem:descriptionOrthInv} that $\Orth_G(X)$ has signature $(0,20)$ and discriminant $A_{\Orth_G(X)}\cong \left(\frac{\IZ}{p\IZ}\right)^2$ and contains no $(-2)$-classes. Furthermore, $\Inv_G(X)$ has signature $(3,0)$ and discriminant $A_{\Inv_G(X)}\cong \frac{\IZ}{2\IZ}\oplus\left(\frac{\IZ}{p\IZ}\right)^2$. We can deduce the isometry class of the invariant lattice as follows. From the classification of Brandt--Intrau~\cite{BI} of positive definite ternary quadratic forms we find that there are two possibilities for the invariant lattice $\Inv_G(X)$, given
 by the following Gram matrices:
$$
A\coloneqq\left(\begin{matrix}2 & 1 & 0 \\ 1 & 6 & 0 \\ 0 & 0 & 22\end{matrix}\right),
\quad B\coloneqq\left(\begin{matrix} 6 & 2 & 2\\2 & 8 & -3\\ 2 & -3 & 8 \end{matrix}\right)
$$
(these two lattices have the same discriminant but different discriminant forms). 

Denote by $Z\subset X\times C$ the universal family, with projections $p,q$ on $X$ and $C$. By Beauville--Donagi~\cite{BD}, the Abel--Jacobi map
$$
\alpha\coloneqq p_\ast q^\ast\colon H^4(C,\IZ)\to H^2(X,\IZ)
$$
is an isomorphism of Hodge structures. Denote by $h\in H^2(C,\IZ)$, resp. $g\in H^2(X,\IZ)$ the hyperplane class of $C$, resp. of $X$ for the Pl\"ucker embedding. Denote by $H^4(C,\IZ)_0$ the primitive cohomology (the orthogonal of $h$ for the intersection form on $C$), and similarly $H^2(X,\IZ)_0$ the orthogonal of $g$ for the Beauville--Bogomolov form, denoted $(-,-)_X$. By Beauville-Donagi~\cite{BD}, the Abel--Jacobi map induces an isometry between $H^4(C,\IZ)_0$ and $H^2(X,\IZ)_0$. It follows from Hassett~\cite{Hassett} that $\disc(H^2(X,\IZ)_0)=3$. Note that $(g,g)_X=6$. By Mongardi~\cite{Mongardi2}, for the special choice of the cubic $C$ with a symplectic automorphism of order $11$ constructed above, one has 
$$
\NS(X)\cong (6)\oplus E_8(-1)^{2}\oplus\left(\begin{matrix} -2 & 1 \\ 1 & -6\end{matrix}\right)
$$
hence $\disc(\NS(X))=2\cdot 3 \cdot 11^2$. Since $\disc(H^2(X,\IZ))=2$, it follows, with the same argument as above, that $\disc(\Trans(X))\in\{2^2\cdot 3\cdot 11^2,3\cdot 11^2\}$. Now consider $\NS(X)_0\coloneqq \NS(X)\cap q^{\perp}\subset H^2(X,\IZ)_0$ and observe that $\Trans(X)=\Trans(X)\cap q^{\perp}\subset H^2(X,\IZ)_0$. One has $\disc(\NS(X)_0)=11^2$ and $\disc(H^2(X,\IZ)_0)=3$ hence
$\disc(\Trans(X))\in\{11^2,3\cdot 11^2\}$. As a consequence, $\disc(\Trans(X))=3\cdot 11^2$.
 Now, observe that since $\Inv_G(X)$ has rank $3$ and $G$ acts symplectically, one has
 $\Trans(X)\subset\Inv_G(X)$ and in fact $\Trans(X)=\Inv_G(X)\cap q^\perp$. From the two 
Gram matrices $A$ and $B$ it is easy to deduce all elements of square~$6$
(there are very few) and to compute their orthogonal and its discriminant. 
One finds that the only possibility giving the right discriminant for $\Trans(X)$ 
is the choice of the matrix $B$
 (compare with Mongardi~\cite{Mongardi2} for a different argument).

%%%%%%%%%%%%%%%%%%%%%%%%%%%%%%%%%%%%%%%%%%%%%%%%%%%%%%%%%%%%%%%%%%
%%%%%%%%%%%%%%%%%%%%%%%%%%%%%%%%%%%%%%%%%%%%%%%%%%%%%%%%%%%%%%%%%%
%%%%%%%%%%%%%%%%%%%%%%%%%%%%%%%%%%%%%%%%%%%%%%%%%%%%%%%%%%%%%%%%%%

\section{Classical methods in Smith theory}

In this section, we study the case when the spectral sequence does not degenerate.
 It happens that one gets bounds for $h^\ast(X^G,\IF_p)$ that are closely 
related to the previous results, where we see that the defect is contained in the
 $p$-torsion of the cohomology of the quotient $X/G$.

%%%%%%%%%%%%%%%%%%%%%%%%%%%%%%%%%%%%%%%%%%%%%%%%%%%%%%%%%%%%%%%%%%

\subsection{Smith exact sequences}

Consider the chain complex $C_*(X)$ of $X$ with coefficients in $\IF_p$ and its subcomplexes $\bar\tau^iC_*(X)$ for $1\leq i\leq p-1$ (with $\bar\sigma=\bar\tau^{p-1})$. The basic tools in Smith theory are the following results, valid for any prime number $p$:

\begin{proposition}\label{prop:SmithSuite}\text{}
\begin{enumerate}
\item\label{SmithSuite1}\cite[Theorem~3.1]{Bredon} For $1\leq i\leq p-1$ there is an exact sequence of complexes:
$$
0\lra \bar\tau^i C_*(X)\oplus C_*(X^G)\overset{\iota}{\lra} C_*(X)\overset{\bar\tau^{p-i}}{\lra} \bar\tau^{p-i}C_*(X)\lra 0
$$
where $\iota$ denotes the sum of the inclusions.

\item\label{SmithSuite2}\cite[p.125]{Bredon} For $1\leq i\leq p-1$ there is an exact sequence of complexes:
$$
0\lra \bar\sigma C_*(X)\overset{\iota}{\lra}\bar\tau^i C_*(X)\overset{\bar\tau}{\lra} \bar\tau^{i+1} C_*(X)\lra 0
$$
where $\iota$ denotes the inclusion.
\item\label{SmithSuite3}\cite[(3.4) p.124]{Bredon} There is an isomorphism of complexes: 
$$
\bar\sigma C_*(X)\cong C_*(X/G,X^G),
$$
where $X^G$ is identified with its image in $X/G$.
\end{enumerate}
\end{proposition}

\begin{proposition}{\cite[p.124 (3.7)]{Bredon}}
\label{prop:ExactSeq}
If $p>2$ then for any $k\in\IN$ there is a commutative diagram of $\IF_p$-vector spaces with exact rows:
$$
\xymatrix{0\ar[r]& \sigma C_k(X)\oplus C_k(X^G)\ar[r]\ar[d]^{\iota\oplus\id}&C_k(X)\ar[r]\ar[d]^{\id}&\tau C_k(X)\ar[r]\ar[d]^{\tau^{p-2}}& 0\\
0\ar[r]& \tau C_k(X)\oplus C_k(X^G)\ar[r]& C_k(X)\ar[r]& \sigma C_k(X)\ar[r]& 0}
$$
\end{proposition}

The \emph{Smith homology groups} are defined by $H_k^{\tau^i}(X)\coloneqq H_k(\bar\tau^i C_*(X))$, the corresponding cohomology 
groups are $H^k_{\tau^i}(X))$, whose dimensions over $\IF_p$ are denoted by $h^k_{\tau^i}(X)$. 
We first give some direct consequences of these sequences.

\begin{lemma}\text{}
\label{lem:H0}
\begin{enumerate}
\item\label{lemma:H0:1} If $X^G\neq\emptyset$ then $h^0_\sigma(X)=0$ and $h^0_\tau(X)=0$.
\item\label{lemma:H0:2} If $X^G=\emptyset$ then $h^0_\sigma(X)=1$ and $h^0_\tau(X)=1$.
\end{enumerate}
\end{lemma}

\begin{proof} First compute for $\sigma$.
By Proposition~\ref{prop:SmithSuite}(\ref{SmithSuite3}), we have $H^0_\sigma(X)\cong H^0(X/G,X^G;\IF_p)$. From the exact sequence of the pair $(X/G,X^G)$:
$$
0\lra C_*(X^G)\lra C_*(X/G)\lra C_*(X/G,X^G)\lra 0
$$
one gets the exact sequence in cohomology:
$$
0\lra H^0(X/G,X^G;\IF_p)\lra H^0(X/G,\IF_p)\xrightarrow{\iota^*} H^0(X^G,\IF_p)\lra\cdots
$$
where $\iota^*$ is induced by the inclusion $\iota\colon X^G\hookrightarrow X/G$. Since $X$ is connected, $X/G$ is also connected. If $X^G\neq\emptyset$, one has that $\iota^*\neq 0$ and $H^0(X/G,\IF_p)\cong\IF_p$ 
so $H^0_\sigma(X)=0$. Otherwise $H^0(X/G,X^G;\IF_p)\cong\IF_p$.

We now compute for $\tau$. Assume that $X^G\neq\emptyset$. By Proposition~\ref{prop:SmithSuite}(\ref{SmithSuite2}) one
 gets for all $1\leq i\leq p-1$ an isomorphism
 $H^0_{\tau^i}(X)\cong H^0_{\tau^{i+1}}(X)$ so $H^0_\tau(X)=H^0_\sigma(X)=0$. If $X^G=\emptyset$, by Proposition~\ref{prop:SmithSuite}(\ref{SmithSuite1}) 
one gets the exact sequence:
$$
0\lra H^0_\tau(X)\lra H^0(X,\IF_p)\xrightarrow{\bar\sigma^*} H^0_\sigma(X)\lra\cdots
$$
We show that $\bar\sigma^*=0$. By Proposition~\ref{prop:SmithSuite}(\ref{SmithSuite1}) one has an exact sequence:
$$
0\lra H^0_\sigma(X)\xrightarrow{\iota^*} H^0(X,\IF_p)\xrightarrow{\bar\tau^*} H^0_\tau(X)\lra\cdots
$$
so $\iota^*$ is injective. The composition $H^0(X,\IF_p)\xrightarrow{\bar\sigma^*} H^0_\sigma(X)\xrightarrow{\iota^*}H^0(X,\IF_p)$
is the action of $\bar\sigma\in\IF_p[G]$. Observe that the action of $g\in G$ on $H^0(X,\IF_p)\cong\IF_p$ is trivial
since $\IF_p$ has no order $p$ automorphism, so $\bar\sigma$ acts trivially on $H^0(X,\IF_p)$, that is $\bar\sigma^*\circ\iota^*=0$. Since $\iota^*$ is injective, this implies that $\bar\sigma^*=0$. It follows that $H^0_\tau(X)\cong H^0(X,\IF_p)\cong\IF_p$. 
\end{proof}

\begin{lemma}
\label{lem:Hodd(X)=0}
Assume that $X$ is even-dimensional, $H^\text{odd}(X,\IF_p)=0$ and $X^G\neq\emptyset$. Set $2d\coloneqq \dim_\IR X$. Then:
\begin{enumerate}
\item\label{lemHodd1} $h^1_\tau(X)=h^1_\sigma(X)=h^0(X^G\IF_p)-1$.
\item\label{lemHodd2} For $0\leq k\leq d-1$, one has $h^{2k+1}_\tau(X)=h^{2k+1}_\sigma(X)$.
\item\label{lemHodd3} If $\dim_\IR X^G\leq\dim_\IR X-2$ then:
$$
h^{2d-1}_\sigma(X)=h^{2d-1}_\tau(X)=h^{2d}_\sigma(X)=h^{2d}_\tau(X)=1.
$$
\end{enumerate}
\end{lemma}

\begin{proof}\text{} 
\par{(\ref{lemHodd1})} Proposition~\ref{prop:SmithSuite}(\ref{SmithSuite1}) for $i=p-1$ and Lemma~\ref{lem:H0} give the exact sequence:
$$
0\lra H^0(X,\IF_p)\lra H^0_\tau(X)\oplus H^0(X^G,\IF_p)\lra H^1_\sigma(X)\lra 0
$$
that implies $h^1_\sigma(X)=h^0(X^G)-1$. By interchanging the roles of $\tau$ and $\sigma$ one gets the second equality.

\par{(\ref{lemHodd2})} Similarly, one gets for $0\leq k\leq d-1$ an exact sequence:
$$
0\rightarrow \substack{H^{2k-1}_\sigma(X)\\\oplus\\ H^{2k-1}(X^G,\IF_p)}\rightarrow H^{2k}_\tau(X) \rightarrow H^{2k}(X,\IF_p)\rightarrow
\substack{H^{2k}_\sigma(X)\\\oplus\\ H^{2k}(X^G,\IF_p)}\rightarrow H^{2k+1}_\tau(X)\rightarrow 0  
$$
that implies the equality:
$$
h^{2k-1}_\sigma(X)+h^{2k-1}(X^G,\IF_p)-h^{2k}_\tau(X)+h^{2k}(X,\IF_p)-h^{2k}_\sigma(X)-h^{2k}(X^G,\IF_p)+h^{2k+1}_\tau(X)=0.
$$
By interchanging the roles of $\tau$ and $\sigma$ one gets a second equality on dimensions. Substracting these equalities, one finally obtains:
$$
h^{2k-1}_\sigma(X)-h^{2k-1}_\tau(X)=h^{2k+1}_\sigma(X)-h^{2k+1}_\tau(X).
$$
Using (\ref{lemHodd1}) one concludes.

\par{(\ref{lemHodd3})} Clearly $h^k_\tau(X)=0=h^k_\sigma(X)$ for $k>2d$. Proposition~\ref{prop:SmithSuite}(\ref{SmithSuite1}) for $i=p-1$ gives the exact sequence:
$$
0\lra H^{2d-1}_\sigma(X)\lra H^{2d}_\tau(X)\lra H^{2d}(X,\IF_p)\lra H^{2d}_\sigma(X)\lra 0
$$
yielding the equality $h^{2d-1}_\sigma(X)-h^{2d}_\tau(X)+1-h^{2d}_\sigma(X)=0$.
Proposition~\ref{prop:SmithSuite}(\ref{SmithSuite3}) and the exact sequence of the pair $(X/G,X^G)$ gives $h^{2d}_\sigma(X)=1$ so $h^{2d-1}_\sigma(X)=h^{2d}_\tau(X)$. By interchanging the roles of $\sigma$ and $\tau$ one gets an exact sequence:
$$
0\lra H^{2d-1}_\tau(X)\lra H^{2d}_\sigma(X)\lra H^{2d}(X,\IF_p)\lra H^{2d}_\tau(X)\lra 0
$$
that implies that $h^{2d}_\tau(X)\leq 1$. Proposition~\ref{prop:SmithSuite}(\ref{SmithSuite2}) for $i=1$ gives an exact sequence:
$$
H^{2d-1}_\tau(X)\lra H^{2d-1}_\sigma(X)\lra H^{2d}_{\tau^2}(X)\lra H^{2d}_\tau(X)\lra H^{2d}_\sigma(X)\lra 0
$$
that implies that $h^{2d}_\tau(X)\geq 1$. Finally $h^{2d}_\tau(X)=1$ and one conludes with (\ref{lemHodd2}).
\end{proof}

Using similar arguments, one can show the following result in the case where the fixed locus is empty:

\begin{lemma}
\label{lem:XG=0}
Assume that $X$ is even-dimensional ($2d\coloneqq \dim_\IR X)$, ${H^\text{odd}(X,\IF_p)=0}$  and $X^G=\emptyset$. Then:
\begin{enumerate}
\item\label{lemXG=01} $h^1_\tau(X)=h^1_\sigma(X)=1$.
\item\label{lemXG=02} For $0\leq k\leq d-1$, one has $h^{2k+1}_\tau(X)=h^{2k+1}_\sigma(X)=h^{2k+1}(X/G,\IF_p)$.
\item\label{lemXG=03} $h^{2d-1}_\sigma(X)=h^{2d-1}_\tau(X)=h^{2d}_\sigma(X)=h^{2d}_\tau(X)=1$.
\end{enumerate}
\end{lemma}

\subsection{A refinement of the Borel-Swan inequality}

\begin{proposition}\label{prop:BorelSwan_p=2} 
Assume that $p=2$, $X$ is even-dimensional ($2d\coloneqq \dim_\IR X$), $H^\even(X,\IZ)$ is torsion-free, $H^\text{odd}(X,\IZ)=0$, $X^G\neq\emptyset$ and $\dim_\IR X^G\leq 2d-2$. 
Then: 
$$
h^*(X^G,\IF_2)\leq h^*(X)-2\sum_{k=1}^{d-1} \aG^{2k}(X)
$$
with equality if $H^\odd(X/G,\IF_2)=0$.
\end{proposition}

\begin{proof}
For $1\leq k\leq d-1$, Proposition~\ref{prop:SmithSuite}(\ref{SmithSuite1}) gives exact an sequence:
$$
0\rightarrow \substack{H^{2k-1}_\tau(X)\\\oplus\\H^{2k-1}(X^G,\IF_2)}\rightarrow H^{2k}_\tau(X)\overset{\alpha_{2k}}{\rightarrow} H^{2k}(X,\IF_2)\overset{\beta_{2k}}{\rightarrow} \substack{H^{2k}_\tau(X)\\\oplus\\H^{2k}(X^G,\IF_2)}\overset{\gamma}{\rightarrow}H^{2k+1}_\tau(X)\rightarrow 0
$$
Using $\im(\alpha_{2k})=\Ker(\beta_{2k})$, this exact sequence cuts into two smaller exact sequences and taking the dimensions one gets the equations:
$$
\left\{
\begin{aligned}
h^{2k-1}_\tau(X)+h^{2k-1}(X^G,\IF_2)-h^{2k}_\tau(X)+\dim\im(\alpha_{2k})&=0\\
\dim\im(\alpha_{2k})-h^{2k}(X)+h^{2k}_\tau(X)+h^{2k}(X^G,\IF_2)-h^{2k+1}_\tau(X)&=0
\end{aligned}
\right.
$$
Summing up these equations, adding the contributions for $1\leq k\leq d-1$ and using Lemma~\ref{lem:Hodd(X)=0}(\ref{lemHodd1}) one gets:
$$
h^*(X^G,\IF_2)=h^*(X)-2\sum_{k=1}^{d-1}\dim\im(\alpha_{2k}).
$$
Denote the components by $\beta_{2k}=\beta'_{2k}\oplus\beta''_{2k}$ and $\gamma=(\gamma',\gamma'')$. Observe that: 
$$
\alpha_{2k}\circ\beta'_{2k}\colon H^{2k}(X,\IF_2)\to H^{2k}(X,\IF_2)
$$ 
is the multiplication by $\bar\tau$. For short, we put $\bar\tau_{2k}\coloneqq \bar\tau_{|H^{2k}(X,\IF_2)}$ and we have $\im(\bar\tau_{2k})\subset\im(\alpha_{2k})$. By Corollary~\ref{cor:linkZp}, $\dim\im(\bar\tau_{2k})=\aG^{2k}(X)$ so we get the expected inequality.

Take $x\in\im(\alpha_{2k})$ and write $x=\alpha_{2k}(y)$ with $y\in H^{2k}_\tau(X)$. Observe that $H^*(X/G,X^G;\IF_2)\cong H^*_\tau(X)$ by Proposition~\ref{prop:SmithSuite}(\ref{SmithSuite3}) since $\bar\tau=\bar\sigma$, so that $\gamma''$ also 
appears as the coboundary morphism of the exact sequence of the pair $(X/G,X^G)$:
$$
\cdots\rightarrow H^{2k}(X^G,\IF_2)\overset{\gamma''}{\rightarrow} H^{2k+1}_\tau(X)\overset{\eta}{\rightarrow} H^{2k+1}(X/G,\IF_2)\rightarrow\cdots
$$

Assume that $H^{2k+1}(X/G,\IF_2)=0$. Then $\gamma''$ is surjective, so there exists an element ${y'\in H^{2k}(X^G,\IF_2)}$ such that $\gamma''(y')=\gamma(y)$. This gives $\gamma(y-y')=0$ so there exists $z\in H^{2k}(X,\IF_2)$ such that ${\beta_{2k}(z)=y-y'}$. In particular $\beta'_{2k}(z)=y$. Then $x=\alpha_{2k}\beta'_{2k}(z)=\bar\tau_{2k}(z)$ giving the equality $\im(\bar\tau_{2k})=\im(\alpha_{2k})$ and we conclude as before.
\end{proof}

\begin{remark}
If $H^\odd(X/G,\IF_2)\neq 0$, the defect in this inequality can be completely understood by the second inequality:
$$
h^*(X)-2\sum_{k=1}^{d-1} \aG^{2k}(X)-2\sum_{k=1}^{d-1} h^{2k+1}(X/G,\IF_2)\leq h^*(X^G,\IF_2).
$$
To prove this inequality, we keep the notation of the proof, assuming now that $H^{2k+1}(X/G,\IF_2)\neq 0$. Consider the map:
$$
\varphi\colon H^{2k}_\tau(X)\lra H^{2k+1}(X/G,\IF_2),\quad y\mapsto \eta\gamma(y).
$$
If $\varphi(y)=0$, there exists $y'\in H^{2k}(X^G,\IF_2)$ such that $\gamma(y)=\gamma''(y')$ so $\gamma(y-y')=0$ and as above there exists $z\in H^{2k}(X,\IF_2)$ such that $\beta_{2k}'(z)=y$. This shows that $\alpha_{2k}(y)\in\im(\bar\tau_{2k})$. 
Conversely, if $y\in H^{2k}_\tau(X)$ is such that $\alpha_{2k}(y)\in\im(\bar\tau_{2k})$, then write $\alpha_{2k}(y)=\alpha_{2k}\beta'_{2k}(z)$ with $z\in H^{2k}(X,\IF_2)$. Setting $y'\coloneqq \beta'_{2k}(z)$, one has $y-y'\in\Ker(\alpha_{2k})$. Writting $\beta_{2k}(z)=\beta'_{2k}(z)+\beta''_{2k}(z)$ and applying $\gamma$ one gets:
$$
\gamma\beta'_{2k}(z)=\gamma'\beta'_{2k}(z)=-\gamma''\beta''_{2k}(z)
$$
so $\eta\gamma(y')=0$. This shows that $\alpha_{2k}^{-1}(\im(\bar\tau_{2k}))=\Ker(\varphi)+\Ker(\alpha_{2k})$. In particular, the map $\alpha_{2k}$ induces a 
surjection:
$$
H^{2k}_\tau(X)/\Ker(\varphi)\twoheadrightarrow \im(\alpha_{2k})/\im(\bar\tau_{2k}).
$$
Computing the dimensions, one gets:
$$
\dim\im(\alpha_{2k})-\dim\im(\bar\tau_{2k})\leq \dim\left(H^{2k}_\tau(X)/\Ker(\varphi)\right)\leq h^{2k+1}(X/G).
$$
We thus proved the inequalities:
$$
\dim\im(\bar\tau_{2k})\leq\dim\im(\alpha_{2k})\leq \dim\im(\bar\tau_{2k})+h^{2k+1}(X/G).
$$
and we conclude as above.\end{remark}

\begin{proposition}\label{prop:BorelSwan_pqqe} 
Assume that $3\leq p\leq 19$, $X$ is even-dimensional ($2d\coloneqq \dim_\IR X$), $H^\even(X,\IZ)$ is torsion-free, $H^\text{odd}(X,\IZ)=0$, $X^G\neq\emptyset$ and $\dim_\IR X^G\leq 2d-2$. Then: 
$$
h^*(X^G,\IF_p)\leq h^*(X)-2\sum_{k=1}^{d-1} \aG^{2k}(X)-(p-2)\sum_{k=1}^{d-1} \mG^{2k}(X)
$$
with equality if  $H^\odd(X/G,\IF_p)=0$.
\end{proposition}

\begin{proof}
For $1\leq k\leq d-1$, Proposition~\ref{prop:SmithSuite}(\ref{SmithSuite1}) with $i=1$ gives exact an sequence:
$$
0\rightarrow \substack{H^{2k-1}_\tau(X)\\\oplus\\H^{2k-1}(X^G,\IF_p)}\rightarrow H^{2k}_\sigma(X)\overset{\alpha_{2k}}{\rightarrow} H^{2k}(X,\IF_p)\overset{\beta_{2k}}{\rightarrow} \substack{H^{2k}_\tau(X)\\\oplus\\H^{2k}(X^G,\IF_p)}\overset{\gamma}{\rightarrow} H^{2k+1}_\sigma(X)\rightarrow 0
$$
Using $\im(\alpha_{2k})=\Ker(\beta_{2k})$ one gets the equations:
$$
\left\{
\begin{aligned}
\dim\im(\alpha_{2k})-h^{2k}_\sigma(X)+h^{2k-1}_\tau(X)+h^{2k-1}(X^G,\IF_p)&=0\\
\dim\im(\alpha_{2k})-h^{2k}(X)+h^{2k}_\tau(X)+h^{2k}(X^G,\IF_p)-h^{2k+1}_\sigma(X)&=0
\end{aligned}
\right.
$$
Summing up these equations, adding the contributions for $1\leq k\leq d-1$ and using Lemma~\ref{lem:Hodd(X)=0} one gets:
$$
h^*(X^G,\IF_2)=h^*(X)+\sum_{k=1}^{d-1} (h^{2k}_\sigma(X)-h^{2k}_\tau(X))-2\sum_{k=1}^{d-1}\dim\im(\alpha_{2k}).
$$
Exchanging the roles of $\tau$ and $\sigma$ (Proposition~\ref{prop:SmithSuite}(\ref{SmithSuite1}) with $i=p-1$) one gets a similar exact sequence, where we denote by $\widetilde{\alpha},\widetilde{\beta}$ the corresponding maps. The same computation gives:
$$
\dim\im(\widetilde{\alpha}_{2k})-h^{2k}_\tau(X)+h^{2k-1}_\sigma(X)+h^{2k-1}(X^G,\IF_p)=0
$$ 
so we get the relation:
$$
\dim\im(\alpha_{2k})-\dim\im(\widetilde{\alpha}_{2k})=h^{2k}_\sigma(X)-h^{2k}_\tau(X).
$$
This gives:
$$
h^*(X^G,\IF_2)=h^*(X)-\sum_{k=1}^{d-1}\dim\im(\alpha_{2k}) -\sum_{k=1}^{d-1}\dim\im(\widetilde{\alpha}_{2k}).
$$
Denote the components by $\beta_{2k}=\beta'_{2k}\oplus\beta''_{2k}$ and $\widetilde{\beta}_{2k}=\widetilde{\beta}'_{2k}\oplus\widetilde{\beta}''_{2k}$. Note that $\alpha_{2k}\circ\widetilde{\beta}'_{2k}$ is the multiplication by $\bar\sigma$ in $H^{2k}(X,\IF_p)$, denoted $\bar\sigma_{2k}$, whereas $\widetilde{\alpha}_{2k}\circ\beta'_{2k}$ is the multiplication by $\bar\tau$ in $H^{2k}(X,\IF_p)$, denoted $\bar\tau_{2k}$. This shows that $\im(\bar\sigma_{2k})\subset\im(\alpha_{2k})$ and $\im(\bar\tau_{2k})\subset\im(\widetilde{\alpha}_{2k})$. By Corollary~\ref{cor:linkZp}, one has $\dim\im(\bar\sigma_{2k})=\aG^{2k}(X)$, and by Proposition~\ref{prop:decompCurtisReiner} and  Corollaries~\ref{cor:linkZp}~\&~\ref{cor:linkmGlp} 
one has:
\begin{align*}
\dim\im(\bar\tau_{2k})&=h^{2k}(X)-\dim H^{2k}(X,\IF_p)^G\\
&=(p-1)\ell_p^{2k}(X)+(p-2)\ell_{p-1}^{2k}(X)\\
&=\aG^{2k}(X)+(p-2)\mG^{2k}(X),
\end{align*}
hence the expected inequality.

Assume that $H^\odd(X/G,\IF_p)=0$.
For $1\leq k\leq d-1$, Proposition~\ref{prop:ExactSeq} gives a commutative diagram with exact rows:
$$
\xymatrix{H^{2k}_\sigma(X)\ar[r]^-{\alpha_{2k}}\ar[d]^{\bar\tau^{p-2}_{2k}} & H^{2k}(X,\IF_p)\ar[r]^-{\beta_{2k}}\ar@{=}[d] & H^{2k}_\tau(X)\oplus H^{2k}(X^G,\IF_p)\ar[r]^-{\gamma}\ar[d]^{\iota^*\oplus\id} & H^{2k+1}_\sigma(X)\ar[r]\ar[d]^{\bar\tau^{p-2}_{2k+1}} & 0\\
H^{2k}_\tau(X)\ar[r]^-{\widetilde{\alpha}_{2k}} & H^{2k}(X,\IF_p)\ar[r]^-{\widetilde{\beta}_{2k}} & H^{2k}_\sigma(X)\oplus H^{2k}(X^G,\IF_p)\ar[r]^-{\delta} & H^{2k+1}_\tau(X)\ar[r] & 0
}
$$ 
We first show that $\bar\tau^{p-2}_{2k+1}$ is injective by a diagram chasing. Denote by $\gamma=(\gamma',\gamma'')$ the components. As previously 
observed, since $H^{2k+1}(X/G,\IF_p)=0$ $\gamma''$ is surjective. Let $x\in H^{2k+1}_\sigma(X)$ such that $\bar\tau^{p-2}_{2k+1}(x)=0$. There exists $y\in H^{2k}(X^G,\IF_p)$ such that $\gamma(y)=\gamma''(y)=x$. Considering $y\in H^{2k}(X^G,\IF_p)$ in the second row of the diagram, one gets $\delta(y)=0$ hence there exists $z\in H^{2k}(X,\IF_p)$ such that $\widetilde{\beta}_{2k}(z)=\widetilde{\beta}_{2k}''(z)=y$. Considering $z$ in the first row, one gets $\beta_{2k}(z)=\beta''_{2k}(z)=y$. Hence $x=\gamma(y)=\gamma\beta_{2k}(z)=0$. By Lemma~\ref{lem:Hodd(X)=0}(\ref{lemHodd2}), we deduce that $\bar\tau^{p-2}_{2k+1}$ is an isomorphism.

We deduce that the map denoted $\iota^*$ in the diagram is surjective. Take ${x\in H^{2k}_\sigma(X)}$. The element $\delta(x)$ admits a preimage $y$ by $\bar\tau^{p-2}_{2k+1}$. Since $\gamma''$ is surjective, there exists $z\in H^{2k}(X^G,\IF_p)$ such that $\gamma''(z)=y$. Considering $z$ in the second row, one gets $\delta(z)=\bar\tau^{p-2}_{2k+1}(y)=\delta(x)$, hence $x-z\in\Ker(\delta)=\im(\widetilde{\beta}_{2k})$ so there exists $y$ such that $\widetilde{\beta}_{2k}'(y)=x$ and $\widetilde{\beta}_{2k}''(y)=-z$. Considering $y$ in the first row, one gets $\iota^*\beta_{2k}'(y)=\widetilde{\beta}_{2k}'(y)=x$. 

We deduce that $\im(\alpha_{2k})\subset\im(\bar\sigma_{2k})$. Take $y=\alpha_{2k}(x)$. Considering ${x\in H^{2k}_\sigma(X)}$ in the second row, since $\iota^*$ is surjective there exists $w\in H^{2k}_\tau(X)$ such that ${\iota^*(w)=x}$. Since $\gamma''$ is surjective, there exists $t\in H^{2k}(X^G,\IF_p)$ with $\gamma''(t)=\gamma(w)$, hence $\gamma(w-t)=0$ so there exists $z\in H^{2k}(X,\IF_p)$ such that $\beta_{2k}'(z)=w$ and $\beta''_{2k}(z)=-t$. Considering $z$ in the second row, one gets $
\widetilde{\beta}_{2k}'(z)=\iota^*\beta'_{2k}(z)=\iota^*(w)=x$
so $y=\alpha_{2k}\widetilde{\beta}_{2k}'(z)=\bar\sigma_{2k}(z)$.

To conclude, we show that $\im(\widetilde{\alpha}_{2k})\subset\im(\bar\tau_{2k})$. Take $y=\widetilde{\alpha}_{2k}(x)$. Considering~$x$ in the first row, since $\gamma''$ is surjective there exists $w$ such that ${\gamma(x)=\gamma''(w)=\gamma(w)}$, so there exists $z$ with $\beta_{2k}'(z)=x$ and $\beta_{2k}''(z)=-w$, hence ${y=\widetilde{\alpha}_{2k}\beta_{2k}'(z)=\bar\tau_{2k}(z)}$. 

The expected equality follows. 

\end{proof}

\bibliographystyle{amsplain}
\bibliography{BiblioAutIHS}

\providecommand{\bysame}{\leavevmode\hbox to3em{\hrulefill}\thinspace}
\providecommand{\MR}{\relax\ifhmode\unskip\space\fi MR }
\providecommand{\MRhref}[2]{%
  \href{http://www.ams.org/mathscinet-getitem?mr=#1}{#2}
}
\providecommand{\href}[2]{#2}
\begin{thebibliography}{10}

\bibitem{AlldayPuppe}
C.~Allday and V.~Puppe, \emph{Cohomological methods in transformation groups},
  Cambridge Studies in Advanced Mathematics, vol.~32, Cambridge University
  Press, Cambridge, 1993. \MR{1236839 (94g:55009)}

\bibitem{AST}
M.~Artebani, A.~Sarti, and S.~Taki, \emph{Non-symplectic automorphisms of prime
  order on {$K3$} surfaces}, Math. Z. (to appear) \textbf{268} (2011),
  507--533.

\bibitem{BeauvilleKaehler}
A.~Beauville, \emph{Some remarks on {K}\"ahler manifolds with {$c_{1}=0$}},
  Classification of algebraic and analytic manifolds ({K}atata, 1982), Progr.
  Math., vol.~39, Birkh\"auser Boston, Boston, MA, 1983, pp.~1--26.

\bibitem{Beauvillec1Nul}
\bysame, \emph{Vari\'et\'es {K}\"ahleriennes dont la premi\`ere classe de
  {C}hern est nulle}, J. Differential Geom. \textbf{18} (1983), no.~4, 755--782
  (1984).

\bibitem{BD}
A.~Beauville and R.~Donagi, \emph{La vari\'et\'e des droites d'une hypersurface
  cubique de dimension {$4$}}, C. R. Acad. Sci. Paris S\'er. I Math.
  \textbf{301} (1985), no.~14, 703--706.

\bibitem{Boissiere}
S.~Boissi\`ere, \emph{Automorphismes naturels de l'espace de {D}ouady de points
  sur une surface}, Canad. J. Math. \textbf{64} (2012), 3--23.

\bibitem{BNWS}
S.~Boissi\`ere, M.~Nieper-Wi{\ss}kirchen, and A.~Sarti, \emph{Higher
  dimensional {E}nriques varieties and automorphisms of generalized kummer
  varieties}, J. Math. Pures Appl. \textbf{95} (2011), 553--563.

\bibitem{Borel}
A.~Borel, \emph{Seminar on transformation groups}, With contributions by G.
  Bredon, E. E. Floyd, D. Montgomery, R. Palais. Annals of Mathematics Studies,
  No. 46, Princeton University Press, Princeton, N.J., 1960.

\bibitem{BI}
H.~Brandt and O.~Intrau, \emph{Tabellen reduzierter positiver tern\"arer
  quadratischer {F}ormen}, Abh. S\"achs. Akad. Wiss. Math.-Nat. Kl. \textbf{45}
  (1958), no.~4, 261.

\bibitem{Bredon}
G.~E. Bredon, \emph{Introduction to compact transformation groups}, Academic
  Press, New York, 1972, Pure and Applied Mathematics, Vol. 46.

\bibitem{Brown}
K.~S. Brown, \emph{Cohomology of groups}, Graduate Texts in Mathematics,
  vol.~87, Springer-Verlag, New York, 1994, Corrected reprint of the 1982
  original. \MR{1324339 (96a:20072)}

\bibitem{CurtisReiner}
C.~W. Curtis and I.~Reiner, \emph{Representation theory of finite groups and
  associative algebras}, Wiley Classics Library, John Wiley \& Sons Inc., New
  York, 1988, Reprint of the 1962 original, A Wiley-Interscience Publication.
  \MR{1013113 (90g:16001)}

\bibitem{Deligne}
P.~Deligne, \emph{Th\'eor\`eme de {L}efschetz et crit\`eres de
  d\'eg\'en\'erescence de suites spectrales}, Inst. Hautes \'Etudes Sci. Publ.
  Math. (1968), no.~35, 259--278.

\bibitem{DW}
W.~G. Dwyer and C.~W. Wilkerson, \emph{Smith theory revisited}, Ann. of Math.
  (2) \textbf{127} (1988), no.~1, 191--198.

\bibitem{Elagin}
A.~D. Elagin, \emph{On an equivariant derived category of bundles of projective
  spaces}, Tr. Mat. Inst. Steklova \textbf{264} (2009), no.~Mnogomernaya
  Algebraicheskaya Geometriya, 63--68.

\bibitem{EGL}
G.~Ellingsrud, L.~G{\"o}ttsche, and M.~Lehn, \emph{On the cobordism class of
  the {H}ilbert scheme of a surface}, J. Algebraic Geom. \textbf{10} (2001),
  no.~1, 81--100.

\bibitem{Fujiki}
A.~Fujiki, \emph{On the de {R}ham cohomology group of a compact {K}\"ahler
  symplectic manifold}, Algebraic geometry, {S}endai, 1985, Adv. Stud. Pure
  Math., vol.~10, North-Holland, Amsterdam, 1987, pp.~105--165.

\bibitem{GS}
A.~Garbagnati and A.~Sarti, \emph{Symplectic automorphisms of prime order on
  {$K3$} surfaces}, J. Algebra \textbf{318} (2007), no.~1, 323--350.

\bibitem{Hassett}
B.~Hassett, \emph{Special cubic fourfolds}, Compositio Math. \textbf{120}
  (2000), no.~1, 1--23.

\bibitem{Huybrechts}
D.~Huybrechts, \emph{Compact hyper-{K}\"ahler manifolds: basic results},
  Invent. Math. \textbf{135} (1999), no.~1, 63--113.

\bibitem{Kharlamov}
V.~M. Kharlamov, \emph{The topological type of nonsingular surfaces in {${\bf
  R}P^{3}$} of degree four}, Funct. Anal. Appl. \textbf{10} (1976), no.~4,
  295--304.

\bibitem{Krasnov1}
V.~A. Krasnov, \emph{Harnack-{T}hom inequalities for mappings of real algebraic
  varieties}, Izv. Akad. Nauk SSSR Ser. Mat. \textbf{47} (1983), no.~2,
  268--297.

\bibitem{Krasnov2}
\bysame, \emph{Real algebraic {GM}-manifolds}, Izv. Ross. Akad. Nauk Ser. Mat.
  \textbf{62} (1998), no.~3, 39--66.

\bibitem{LehnSorger}
M.~Lehn and C.~Sorger, \emph{The cup product of {H}ilbert schemes for {$K3$}
  surfaces}, Invent. Math. \textbf{152} (2003), no.~2, 305--329.

\bibitem{Markman3}
E.~Markman, \emph{The {B}eauville--{B}ogomolov class as a characteristic
  class}, \texttt{arXiv:1105:3223v1}.

\bibitem{MarkmanTorelli}
\bysame, \emph{A survey of {T}orelli and monodromy results for
  holomorphic-symplectic varieties}, \texttt{arXiv:1101.4606v2}.

\bibitem{Markman}
\bysame, \emph{Integral generators for the cohomology ring of moduli spaces of
  sheaves over {P}oisson surfaces}, Adv. Math. \textbf{208} (2007), no.~2,
  622--646.

\bibitem{Markman2}
\bysame, \emph{Integral constraints on the monodromy group of the
  hyper{K}\"ahler resolution of a symmetric product of a {$K3$} surface},
  Internat. J. Math. \textbf{21} (2010), no.~2, 169--223.

\bibitem{Markushevich}
D.~Markushevich, \emph{Rational {L}agrangian fibrations on punctual {H}ilbert
  schemes of {$K3$} surfaces}, Manuscripta Math. \textbf{120} (2006), no.~2,
  131--150.

\bibitem{MasleyMontgomery}
J.~M. Masley and H.~L. Montgomery, \emph{Cyclotomic fields with unique
  factorization}, J. Reine Angew. Math. \textbf{286/287} (1976), 248--256.

\bibitem{Mongardi2}
G.~Mongardi, \emph{On symplectic automorphisms of hyperk\"ahler fourfolds},
  \texttt{arXiv:1112.5073v3}.

\bibitem{Mongardi1}
\bysame, \emph{Symplectic involutions on deformations of {$K3^{[2]}$}},
  \texttt{arXiv:1107.2854}.

\bibitem{Nakajima}
H.~Nakajima, \emph{Heisenberg algebra and {H}ilbert schemes of points on
  projective surfaces}, Ann. of Math. (2) \textbf{145} (1997), no.~2, 379--388.

\bibitem{NW}
M.~Nieper-Wi{\ss}kirchen, \emph{Chern numbers and {R}ozansky-{W}itten
  invariants of compact hyper-{K}\"ahler manifolds}, World Scientific
  Publishing Co. Inc., River Edge, NJ, 2004.

\bibitem{Nikulin}
V.~V. Nikulin, \emph{Finite groups of automorphisms of {K}\"ahlerian {$K3$}
  surfaces}, Trudy Moskov. Mat. Obshch. \textbf{38} (1979), 75--137.

\bibitem{Nikulinfactor}
\bysame, \emph{Integral symmetric bilinear forms and some of their geometric
  applications}, Izv. Akad. Nauk SSSR Ser. Mat. \textbf{43} (1979), no.~1,
  111--177.

\bibitem{OGrady}
Kieran~G. O'Grady, \emph{Irreducible symplectic 4-folds numerically equivalent
  to {$(K3)^{[2]}$}}, Commun. Contemp. Math. \textbf{10} (2008), no.~4,
  553--608.

\bibitem{OS}
K.~Oguiso and S.~Schr{\"o}er, \emph{Enriques manifolds}, J. Reine Angew. Math.
  (to appear) (2011).

\bibitem{QinWang}
Z.~Qin and W.~Wang, \emph{Integral operators and integral cohomology classes of
  {H}ilbert schemes}, Math. Ann. \textbf{331} (2005), no.~3, 669--692.

\bibitem{Verbitsky}
M.~Verbitsky, \emph{Cohomology of compact hyper-{K}\"ahler manifolds and its
  applications}, Geom. Funct. Anal. \textbf{6} (1996), no.~4, 601--611.

\end{thebibliography}

\end{document}